\documentclass[sn-mathphys-num]{sn-jnl}

\usepackage{graphicx}%
\usepackage{multirow}%
\usepackage{amsmath,amssymb,amsfonts}%
\usepackage{amsthm}%
\usepackage{mathrsfs}%
\usepackage[title]{appendix}%
\usepackage{xcolor}%
\usepackage{textcomp}%
\usepackage{manyfoot}%
\usepackage{booktabs}%
\usepackage{algorithm}%
\usepackage{algorithmicx}%
\usepackage{algpseudocode}%
\usepackage{listings}%
\usepackage{geometry}
\theoremstyle{thmstyleone}%
\newtheorem{theorem}{Theorem}
\newtheorem{proposition}[theorem]{Proposition}%

\theoremstyle{thmstyletwo}%
\newtheorem{example}{Example}%
\newtheorem{remark}{Remark}%
\newtheorem{lemma}{Lemma}
\newtheorem{corollary}{Corollary}%

\theoremstyle{thmstylethree}%
\newtheorem{definition}{Definition}%

\DeclareMathOperator*{\Limsup}{Lim\,sup}

\begin{document}

\title[Characterization of Highly Robust Solutions in Multi-Objective Programming in Banach Spaces]{Characterization of Highly Robust Solutions in Multi-Objective Programming in Banach Spaces}




\author[1]{\fnm{Morteza} \sur{Rahimi}}

\author[2]{\fnm{Majid} \sur{Soleimani-damaneh}}

\affil[1]{
\orgdiv{Department of Mathematics},
\orgname{University of Antwerp},
\orgaddress{\city{Antwerp},
\country{Belgium}},
\href{mailto:morteza.rahimi@uantwerp.be}{morteza.rahimi@uantwerp.be}}

\affil[2]{
\orgdiv{School of Mathematics, Statistics and Computer Science, College of Science},
\orgname{University of Tehran},
\orgaddress{\city{Tehran},
\country{Iran}},
\href{mailto:soleimani@khayam.ut.ac.ir}{soleimani@khayam.ut.ac.ir}}


\abstract{This paper delves into the challenging issues in uncertain multi-objective optimization, where uncertainty permeates nonsmooth nonconvex objective and constraint functions. In this context, we investigate highly robust (weakly efficient) solutions, a solution concept defined by efficiency across all scenarios. Our exploration reveals important relationships between highly robust solutions and other robustness notions, including set-based and worst-case notions, as well as connections with proper and isolated efficiency. Leveraging modern techniques from variational analysis, we establish necessary and sufficient optimality conditions for these solutions. Moreover, we explore the robustness of multi-objective optimization problems in the face of various uncertain sets, such as ball, ellipsoidal, and polyhedral sets.}




\maketitle

\section{Introduction}\label{section1}

In diverse domains such as engineering, business, management, and more, decision problems are frequently characterized by conflicting objectives and simultaneous uncertainties arising from imperfect models, unknown data, incomplete information, perturbations, and the dynamic nature of the global environment. This necessitates sophisticated approaches to optimization, particularly in the context of Multi-objective Optimization Problems (MOPs). Acknowledging the ubiquity of uncertainty, various methodologies grounded in probabilistic, possibilistic, and deterministic perspectives have been developed \cite{ber,zam}. Among these, a deterministic approach known as robust optimization has emerged as a powerful framework to analyze mathematical programming problems in the face of data uncertainty \cite{ben-b,ber}.

The focal point of this paper is the exploration of robustness in the context of multi-objective optimization. This critical consideration aims to empower decision-makers by facilitating the identification of efficient solutions that exhibit resilience against perturbations in the problem data, specifically concerning objective and constraint functions. Researchers have extensively explored this field from various aspects, delving into both single- and multi-objective optimization; see, for example, \cite{ave, ben-b, ber, bon, soy} for single-objective problems and \cite{bok, chu-l, chu-o, chu-r, deb, ehr-m, geo, gob-r, ide-c, ide-t, ide-r, kur, rah-r, zam} for multi-objective problems. Practical applications of robust optimization are also covered in many works such as \cite{ben-ro, cal, fli}.

Soyster \cite{soy} introduced the first robustness technique for solving linear single-objective uncertain problems. This approach, pioneered by Ben-Tal and Nemirovski \cite{ben-rsu, ben-ro}, was extended as minmax robust optimization. Bertsimas et al. \cite{ber} conducted an insightful exploration of robustness in single-objective optimization facing data uncertainty in the constraints. Their work established valuable connections between adaptable models designed for decision-making across multiple stages. Calafiore \cite{cal} presented an efficient approach to dealing with a financial decision-making problem under uncertainty, obtaining optimal robust portfolios that are worst-case optimal from the mean-risk perspective.

As the frontier of research expanded, the integration of robustness with multiple criteria decision-making led to the emergence of robust multi-objective optimization, a field that has garnered significant attention in the last decade. Pioneering work by Deb and Gupta \cite{deb} laid the foundation for this field, where robustness is characterized by the efficiency of a solution concerning the mean representation of objective function values in its vicinity, rather than the original objective functions. Goberna et al. \cite{gob-r} studied multi-objective linear semi-infinite programming problems under constraint data uncertainty. However, various robustness notions exist in the multi-objective framework under uncertainty in objective functions, such as worse-case, set-based, norm-based, and so on.

Worst-case robustness, as an extension of minmax robustness, has appeared in recent works by Kuroiwa and Lee \cite{kur}, Bokrantz and Fredriksson \cite{bok}, Ehrgott et al. \cite{ehr-m}, Fliege and Werner \cite{fli}, and Chuong \cite{chu-o, chu-l, chu-r}. Furthermore, Bokrantz and Fredriksson \cite{bok}, Ehrgott et al. \cite{ehr-m}, Ide and K\"{o}bis \cite{ide-c}, Ide et al. \cite{ide-t}, and Ide and Sch\"{o}bel \cite{ide-r} studied other notions of robustness defined by set relations. In some other works, various researchers, like Georgiev et al. \cite{geo}, Zamani et al. \cite{zam}, and Rahimi and Soleimani-damaneh \cite{rah-r}, investigated norm-based robustness.

Another solution concept in robust optimization is highly robust efficiency/ optimality, which refers to efficient/optimal solutions that maintain their efficiency/ optimality in all scenarios. These solutions, first defined by Bitran \cite{bir} as necessarily efficient solutions, represent a significant advancement in addressing uncertainty while ensuring efficiency.

Ide and Sch\"{o}bel \cite{ide-r} initially introduced the notion of high robustness in multi-objective optimization, and Goberna et al. \cite{gob-g} and Dranichak and Wiecek \cite{dra} further studied these solutions. For uncertain convex MOPs with objective-wise uncertainty, Goberna et al. \cite{gob-g} calculated the radius of high robustness (for weak efficiency) and provided sufficient conditions for the existence of these solutions. In this framework, Dranichak and Wiecek \cite{dra} investigated these solutions in a linear setting.

This paper delves into the intricate realm of uncertain Multi-Objective Problems (MOPs), where objective-wise uncertainty permeates both nonsmooth nonconvex objective and constraint functions. Despite the extensive literature on robust optimization and multi-objective optimization individually, the optimization community has not yet extensively explored the intersection of these domains with highly robust efficient solutions in the context of uncertain nonsmooth nonconvex optimization.

One of our objectives is to leverage advanced techniques from variational analysis and generalized differentiation to establish necessary and sufficient optimality conditions for (local) highly robust (weakly) efficient solutions of uncertain MOPs. In subsequent sections, we investigate the relationships between highly robust solutions and other concepts, including set-based and worst-case robustness, proper efficiency, and isolated efficiency. Through this exploration, necessary and sufficient optimality conditions for local highly robust (weakly) efficient solutions are established. These conditions are formulated in terms of Mordukhovich and Clarke subdifferentials in both Asplund and Banach spaces, presenting a nuanced understanding of robust optimality.

The paper is structured as follows. In section \ref{section2}, we revisit fundamental definitions from variational analysis and present several auxiliary results. The connections between highly robust solutions and other solution concepts in uncertain multi-objective optimization are explored in section \ref{section3}. Section \ref{section4} is dedicated to establishing necessary and sufficient conditions for local highly robust (weakly) efficient solutions. The consideration of highly robustness in Multi-Objective Problems (MOPs) in the presence of various uncertain sets is outlined in section \ref{section5}.


\section{Preliminaries}\label{section2}

This section contains some basic notions and notations which are used in the sequel.
Throughout the paper, we utilize mainly standard notations from variational and nonsmooth analysis; see \cite{cla-f,mor-b1,roc-v}.
Except where otherwise stated, all spaces are assumed to be real Banach spaces.
Given a real Banach space $X$, we indicate the norm of the space by $\Vert \cdot\Vert$ and the canonical pairing between $X$ and
its topological dual $X^*$ by $\langle \cdot , \cdot \rangle$.
The notation $\mathbb{B}(\bar{x};\delta)$ is employed for the open ball centered at $\bar{x}\in X$ with radius $\delta>0$.
Furthermore, the convergence in the weak topology of $X$ and the weak$^*$ topology of $X^*$ are denoted by
$\overset{w}{\longrightarrow}$ and $\overset{w^*}{\longrightarrow}$, respectively.

Given a nonempty set $\Omega\subseteq X$, the notations $int\,\Omega$, $cl\,\Omega$, $cl^*\Omega$, $co\,\Omega$,
and $cone\,\Omega$ stand for the norm topology interior, the norm topology closure, the weak$^*$ topology closure, the convex hull,
and the conic hull of $\Omega$, respectively.
The set $\Omega\subseteq X$ is called a cone, if $\lambda x \in \Omega$ for any $x\in \Omega$ and $\lambda \geq 0$.


Consider the set-valued mapping $F:Y\rightrightarrows X^*$ where $Y$ is a Hausdorff topological vector space.
We say that $F$ is weak$^*$ closed at $y\in Y$ if for any nets $\{y_\nu\}_\nu \subset Y$ and $\{x^*_\alpha\}_\alpha \subset X^*$
with $x_\alpha^* \in F(y_\nu)$ such that $y_\nu \rightarrow y$ and $x^*_\alpha \overset{w^*}{\longrightarrow} x^*$,
one has $x^* \in F(y)$.

\begin{lemma}\label{lem}
Let $S\subseteq Y$ be compact and the set-valued mapping $F:Y\rightrightarrows X^*$  be weak$^*$ closed at each $\bar{y}\in S$.
Then, $F(S):=\bigcup_{y\in S}F(y)$ is weak$^*$ closed.
\end{lemma}
\begin{proof}
Assume that the net $\{x^*_\alpha\}_\alpha \subseteq F(S)$ is weak$^*$ convergent to $\bar{x}^*\in X^*$.
For each $\alpha$, $x^*_\alpha\in F(S)$ implies that there exists $y_\alpha \in S$ such that $x^*_\alpha \in F(y_\alpha)$.
Now, since $S$ is compact, the net $y_\alpha$ has a convergent subnet.
Without loss of generality, assume that $y_\alpha \longrightarrow \bar{y}$ for some $\bar{y}\in S$.
Then, weak$^*$ closedness of $F$ yields $\bar{x}^* \in F(S)$.
\end{proof}

Given a set-valued mapping $F:X\rightrightarrows X^*$ and $\bar{x}\in X$,
the sequential Painlev\'{e}--Kuratowski upper/outer limit of $F$ as $x \rightarrow \bar{x}$
is defined by
$$\displaystyle\Limsup_{x\rightarrow \bar{x}} F(x):=\left\{ x^* \in X^* :~ \exists\, x_\nu \rightarrow \bar{x},~
 \exists\, x^*_\nu \overset{w^*}{\longrightarrow} x^* : ~x^*_\nu \in F(x_\nu),~ \nu\in\Bbb N \right\}.$$


Let $\Omega \subseteq X$ be a nonempty set and $\varepsilon\geq 0$.
The set of all $\varepsilon$-normals to $\Omega$ at $\bar{x} \in\Omega$
is defined by
$$\widehat{N}_\varepsilon (\bar{x};\Omega):=\big\{x^* \in X^* :~ \limsup_{x \overset{\Omega}{\longrightarrow} \bar{x}}
\frac{\langle x^*, x-\bar{x}\rangle}{\Vert x-\bar{x}\Vert} \leq \varepsilon \big\},$$
where $x \overset{\Omega}{\longrightarrow} \bar{x}$ means $x \rightarrow \bar{x}$ with $x \in \Omega$; See \cite{mor-b1}.
For $\varepsilon=0$, this set is called Fr\'{e}chet normal cone to $\Omega$ at $\bar{x}$ and is denoted by
$\widehat{N}(\bar x;\Omega) := \widehat{N}_0(\bar x;\Omega)$.
Furthermore, by taking the sequential Painlev\'{e}--Kuratowski upper limits from $\widehat N_\varepsilon (x;\Omega)$
as $x\overset{\Omega}{\longrightarrow}\bar{x}$ and $\varepsilon \downarrow 0$,
the (basic/limiting) Mordukhovich normal cone to $\Omega$ at $\bar{x}$, denoted by $N(\bar{x};\Omega)$, is derived:
$N(\bar{x};\Omega):=\Limsup_{\stackrel{x\overset{\Omega}{\longrightarrow} \bar x}{\varepsilon\downarrow0}}
\widehat{N}_\varepsilon (x;\Omega).$

If $X$ is an Asplund space (a Banach space whose separable subspaces have separable duals) and $\Omega$ is closed, then
$N(\bar{x};\Omega)$ can be obtained by taking the sequential Painlev\'{e}--Kuratowski upper limits from $\widehat N (x;\Omega)$ as $x\longrightarrow\bar{x}$; see \cite[Theorem 2.35]{mor-b1}.

The weak contingent cone to $\Omega$ at $\bar{x}$, denoted by $T_w(\bar{x};\Omega)$, is defined as
$$T_w(\bar{x};\Omega) := \bigg\{d\in X :~ \exists (\{x_\nu\}_\nu \subseteq \Omega, \{t_\nu\}_\nu \subseteq \mathbb{R}):
~ t_\nu \downarrow 0, ~\frac{x_\nu - \bar{x}}{t_\nu}\overset{w}{\longrightarrow} d \bigg\}.$$
If $\Omega$ is locally convex at $\bar{x}$, i.e., $\Omega \cap \mathcal{B}$ is convex for some neighbourhood
$\mathcal{B}$ of $\bar{x}$, then
$$N(\bar{x};\Omega)=\widehat{N}(\bar{x};\Omega)=\big\{x^*\in X^* :~ \langle x^*, x-\bar{x}\rangle \leq 0,~
\forall x\in \Omega \cap \mathbb{B}(\bar{x};\delta)\big\},$$
see \cite[Proposition 1.5]{mor-b1}. On the other hand, if $X$ is reflexive, then due to \cite[Corollary 1.11]{mor-b1}, we get
$$\widehat{N}(\bar{x};\Omega)=T^*_w(\bar{x};\Omega):=
\big\{x^*\in X^* :~ \langle x^*, d\rangle \leq 0,~\forall d \in T_w(\bar{x};\Omega)\big\}.$$


Let $\varphi:X\rightarrow \overline{\mathbb{R}} := [-\infty,\infty]$ be an extended real-valued function.
The domain and epigraph of $\varphi$ are respectively defined as $dom\,\varphi :=\big\{x\in X:~ \varphi (x)<\infty\big\}$ and
$epi\,\varphi :=\big\{(x,\alpha) \in X \times \mathbb{R} ~:~ \alpha \geq \varphi(x)\big\}.$ We say $\varphi$ is proper if $\varphi(x) >-\infty$ for each $x\in X$ and $dom\,\varphi\neq \emptyset$.

In the rest of this section, let $\varphi$ be proper and $\bar{x}\in dom\,\varphi$.
The Mordukhovich subdifferential (known as general, limiting, and basic first-order subdifferential as well), singular subdifferential,
and Fr\'{e}chet subdifferential (presubdifferential) of
$\varphi$ at $\bar{x}$, denoted by $\partial\varphi(\bar{x})$, $\partial^\infty\varphi(\bar{x})$, and $\hat{\partial}\varphi(\bar{x})$, respectively, are defined as
$$\begin{array}{c}
\partial\varphi(\bar{x}) := \big\{x^* \in X^* :~ (x^*,-1) \in N((\bar{x},\varphi(\bar{x}));epi\,\varphi)\big\},\vspace*{1mm}\\
\partial^\infty\varphi(\bar{x}) := \big\{x^* \in X^* :~ (x^*,0) \in N((\bar{x},\varphi(\bar{x}));epi\,\varphi)\big\},\vspace*{1mm}\\
\hat{\partial}\varphi(\bar{x}) := \big\{x^* \in X^* :~ (x^*,0) \in \widehat{N}((\bar{x},\varphi(\bar{x}));epi\,\varphi)\big\}.
\end{array}
$$
The Clarke's generalized gradient (Clarke's subdifferential) of $\varphi$ at $\bar{x}$,
denoted by $\partial^C \varphi(\bar{x})$, is defined as
$$\partial^C\varphi(\bar{x}) := \big\{x^* \in X^* :~ \langle x^*,d\rangle \leq \varphi^\circ(\bar{x};d),~\forall d\in X\big\},$$
where $\varphi^\circ(\bar{x};d)$ is the Clarke's generalized directional derivative of $\varphi$ at $\bar{x}$ in the direction $d$ defined by $\varphi^\circ(\bar{x};d):=\limsup_{\stackrel{x\rightarrow \bar{x}}{t\downarrow 0}}
\frac{\varphi(x+td)-\varphi(x)}{t}.$
If $\varphi$ is locally Lipschitz at $\bar{x}$, the set $\partial^C\varphi(\bar{x})$ is nonempty.
The locally Lipschitz function $\varphi$ is called regular at $\bar{x}$ in the sense of Clarke if for each $d\in X$, one has
$\varphi^\circ(\bar{x};d)=\lim_{t\downarrow 0}\frac{\varphi(\bar{x}+td)-\varphi(\bar{x})}{t}.$
Invoking \cite[Corollary 1.81]{mor-b1}, if $\varphi$ is locally Lipschitz at $\bar{x}$ with constant $L>0$, then
\begin{equation}\label{eqs}
\partial^\infty\varphi(\bar{x})=\{0\},~\textmd{~and~}~\Vert x^*\Vert \leq L,~~~\forall x^*\in \partial\varphi(\bar{x}),
\end{equation}
and if, in addition, $X$ is an Asplund space, then $\partial\varphi(\bar{x})\neq \emptyset$ by \cite[Theorem 3.57]{mor-b1} and
\begin{equation}\label{eqcm}
\partial^C\varphi(\bar{x})=cl^*co\,\partial\varphi(\bar{x}),
\end{equation}
due to \cite[Corollary 2.25]{mor-b1}.


Let $h: X\times \Theta \rightarrow \overline{\mathbb{R}}$ be a proper extended real-valued
function of two variables $(x,\theta) \in X\times \Theta$ where $\Theta$ is equipped with a Hausdorff topology.
The basic, singular, and Clarke's partial subdifferential of $h$, denoted by $\partial_xh$, $\partial_x^\infty h$, and
$\partial_x^C h$, respectively, are defined by
$$\partial_x h(\bar{x},\bar{\theta}):= \partial h^{\bar{\theta}}(\bar{x}),~~~~~
\partial_x^\infty h(\bar{x},\bar{\theta}):= \partial^\infty h^{\bar{\theta}}(\bar{x}),~~~~~
\partial_x^C h(\bar{x},\bar{\theta}):= \partial^C h^{\bar{\theta}}(\bar{x}),$$
for each $(\bar{x}, \bar{\theta})\in dom\, h$, where $h^{\bar{\theta}}(\cdot):=h(\cdot,\bar{\theta})$.\\%
We say that $h$ is locally Lipschitz in $x$ uniformly in $\theta$ around $\bar{x}$ with modulus $L>0$, if
there exists neighbourhood $\mathcal{B}$ of $\bar{x}$ such that
$$\vert h(x,\theta)-h(y,\theta)\vert \leq L\Vert x-y \Vert,~~~\forall x,y\in \mathcal{B},~~\forall \theta\in \Theta.$$


In $\mathbb{R}^n$, given two vectors $a$ and $b$, we apply the symbols $a^T$ and $a^Tb$ to denote the transpose of $a$ and the standard inner product of $a$ and $b$, respectively. Furthermore, we use the Euclidean norm $\Vert a\Vert :=\sqrt{a^Ta}$.

We close this section by some preliminaries about optimality concepts for MOPs. Let $f:X\longrightarrow \mathbb{R}^p$ be a proper extended real-valued function with $p\geq 2$ and
$$f(x)=(f_1(x),f_2(x),\ldots,f_p(x))^T,~~x\in dom\,f.$$ An MOP is given by
\begin{equation}\label{prb}
\begin{array}{l}
\min\,f(x)~~s.t.~~x\in \Omega
\end{array}
\end{equation}
where $\Omega\subseteq X$ is the set of feasible solutions and $f$ is the objective function.
Considering two vectors $x,y\in \mathbb{R}^p$  with $p\geq 2$, we utilize the ordering relations $\leqq, \leq, <$ as follows:
$$\begin{array}{c}
x \leqq y ~(\text{resp.~} x<y) \Longleftrightarrow x_i \leq y_i~(\text{resp.~} x_i<y_i),~~\forall i\in I:=\{1,2,,\ldots, p\},\vspace*{2mm}\\
x\leq y \Longleftrightarrow x\leqq y \text{~and~} x\neq y.
\end{array}$$
Also, we apply the corresponding cones
$\mathbb{R}^p_>:=\{x\in \mathbb{R}^p :x>0\}$,
$\mathbb{R}^p_\geq:=\{x\in \mathbb{R}^p :x\geq 0\},$ and
$\mathbb{R}^p_\geqq:=\{x\in \mathbb{R}^p :x\geqq 0\}$.

In this case, an element $\bar{x} \in \Omega$ is called an \textit{efficient (resp. weakly efficient, strictly efficient) solution} of (\ref{prb}) if there exists no $x\in \Omega\setminus \{\bar{x} \}$ such that $f(x)\leq f(\bar{x})$ (resp. $f(x) < f(\bar{x})$, $f(x)\leqq f(\bar{x})$).
The vector $\bar{x}$ is called an \textit{isolated efficient solution} \cite{ehr-b,rah-i} of (\ref{prb}) if there exists some $L>0$ such that
$$\max_{i\in I} \{f_i(x)-f_i(\bar{x})\}\geq L\Vert x-\bar{x}\Vert,~~\forall x\in \Omega.$$

The vector $\bar{x}\in \Omega$ is called a \textit{Henig properly efficient solution} of (\ref{prb}) if there exists a convex cone $K$
such that $\mathbb{R}^p_\geq\setminus\{0\} \subseteq int\,K$ and there exists no $x\in\Omega$ with $f(x)- f(\bar{x})\in K\setminus \{0\}$.

In the conditional part of the above definitions, if $\Omega$ is replaced by $\Omega \cap \mathcal{B}$ for some neighbourhood $\mathcal{B}$ of $\bar{x}$, then the corresponding concepts are defined locally.


\section{Highly Robust Solutions in Multi-Objective Optimization}\label{section3}

As a motivating example, we consider a portfolio problem, which is important, widely studied, and challenging in finance. Assume an investor who is going to invest some money, say $E$, in a financial market with $n$ risky assets such as securities and shares. The investor is seeking a portfolio with maximum expected return and minimum financial risk, among other criteria \cite{mar1}. Let the decision variable $x_i$ for $i=1, 2, \ldots, n$ represent the weight of the $i$-th asset. Assuming short-selling is not allowed, a feasible portfolio $x=(x_1, x_2, \ldots , x_n) \in \mathbb{R}^n$ must satisfy basic constraints: $x \geqq 0$ and $\sum_{i=1}^n x_i \leq E$. Additional constraints may be present, such as $\bar{a}_j^T x \leq \bar{b}_j$ for $j \in J$ (where $J$ is some index set). These constraints could correspond to maintaining certain indices, like sustainability and liquidity, within given intervals; see \cite{Kou}. Therefore, the set of feasible portfolios can be expressed as:
$\Omega:= \{x\in \mathbb{R}^n :~ x\geqq 0,~\sum_{i=1}^n x_i\leq E,~ \bar{a}^T_j x \leq \bar{b}_j,~j \in J\}.$
Let $r_i$ be the parameter corresponding to the rate of return of asset $i$, and $r=(r_1,r_2,\ldots,r_n)^T$ be the return vector. Assume $r$ is normally distributed with a mean return $\bar{r}$ and covariance matrix $\Sigma$, i.e., $r \sim \mathcal{N}(\bar{r},\Sigma)$. In this case, for a feasible portfolio $x \in \Omega$, the expected return is $\textbf{E}[r^Tx] = \bar{r}^Tx$, and the investment risk in a mean-variance framework is measured by $\textbf{var}[r^Tx] = x^T\bar{\Sigma}x$ \cite{mar1}. Therefore, the objective functions will be maximizing $\bar{r}^Tx$ and minimizing $x^T\bar{\Sigma}x$ over $x \in \Omega$.

In portfolio models encountered in practical applications, the parameters of the problem may not be precisely known due to factors such as data uncertainty, estimation errors, environmental influences, lack of information, and other sources of variability; refer to \cite{fli, Kou}. Consider $\mathcal{V}_j \subseteq \mathbb{R}^{n+1}$ as an uncertainty set for $(\bar{a}_j, \bar{b}_j)$, where $j \in J$. For each $(a_j, b_j) \in \mathcal{V}_j$, $a_j \in \mathbb{R}^n$ and $b_j \in \mathbb{R}$ are approximations of $\bar{a}_j$ and $\bar{b}_j$, respectively. Hence, the robust feasible set is defined as:
$$\Omega_{\mathcal{V}}:= \{x\in \mathbb{R}^n :~ x\geqq 0,~\sum_{i=1}^n x_i\leq E,~ a^T_j x \leq b_j,
~(a_j,b_j)\in \mathcal{V}_j,~j \in J\}.$$

Each $x\in \Omega_{\mathcal{V}}$ is called a robust feasible solution.
We do not have certain information about return and risk in the future, and these two
factors are uncertain in their essence. In practice, the sets in which these two factors may
change could be estimated by means of the historical data in a time window \cite{Kou}.
The parameters $\bar{r}$ and $\bar{\Sigma}$ are
uncertain, driven from historical data.
Let $r\in \mathbb{R}^n_>$ and $\Sigma\in\Gamma^n$, where $\Gamma^n$ is
the cone of positive semi-definite $n\times n$ matrices, be estimations of $\bar{r}$ and $\bar{\Sigma}$, and
let $\mathcal{U}_1\subseteq \mathbb{R}^n_>$ and $\mathcal{U}_2\subseteq \Gamma^n$ with
$(r,\Sigma) \in \mathcal{U}:=\mathcal{U}_1 \times \mathcal{U}_2$ be the uncertainty sets
for $\bar{r}$ and $\bar{\Sigma}$, respectively. Then, the uncertain portfolio problem is modelled as follows:
$$\begin{array}{lll}
(P_{(r,\Sigma)}):~&\min &(-r^Tx,x^T\Sigma x)\vspace*{1mm}\\
&\,s.t.~&\sum_{i=1}^n x_i\leq E,\vspace*{1mm}\\
&&a^T_j x \leq b_j,~(a_j,b_j)\in \mathcal{V}_j,~j \in J,\vspace*{1mm}\\
&&x_i\geq 0,~i=1, 2,\ldots , n.
\end{array}$$
In this context, robust optimization plays a vital role in computing an appropriate portfolio that optimizes the objectives of $(P_{(r,\Sigma)})$ given the uncertainty sets $\mathcal{U}$ and $\mathcal{V}_j,~j \in J$. Since the objective functions are often in conflict, it is sensible to consider uncertainty sets dedicated to each function. This approach was first introduced by Goberna et al. \cite{gob-g}, who utilized the term \textit{objective-wise uncertainty} to characterize problems where the uncertainty set $\mathcal{U}$ is represented as $\mathcal{U}:=\prod_{i=1}^p \mathcal{U}_i$.


Assume that in the nominal Problem (\ref{prb}) the constraints are described by real-valued functions $g_j,~j\in J:=\{1,2,\ldots,q\}.$
Perturbing each objective function of this nominal problem by adding a linear term,
whose coefficients come from an uncertainty set, lead in the following parameterized MOP:
$$\begin{array}{lll}
(P_{u,v}):~&\min &(f_1(x)-\langle u_1,x\rangle, f_2(x)-\langle u_2,x\rangle,\ldots ,f_p(x)-\langle u_p,x\rangle)^T\vspace*{1mm}\\
&\,s.t.~& g_j(x,v_j)\leq 0,~~j\in J,
\end{array}$$
where $x\in X$ is the decision variable vector, $u_i \in\mathcal{U}_i,~i\in I$, and
$v_j \in\mathcal{V}_j,~j\in J$, are uncertain parameters. Here, $\mathcal{U}_i \subseteq X^*,~i\in I$, and $\mathcal{V}_j,~j\in J$,
equipped with a Hausdorff topology, are nonempty uncertainty sets. Furthermore, $u=(u_1,u_2,\ldots,u_p)$ and $v=(v_1,v_2,\ldots,v_q)$ are called scenarios and vary in uncertainty sets
$\mathcal{U}:=\prod_{i=1}^p \mathcal{U}_i$ and $\mathcal{V}:=\prod_{j=1}^q \mathcal{V}_j$, respectively.
Also,  $g_j:X\times \mathcal{V}_j\longrightarrow \overline{\mathbb{R}},~j\in J,$ are proper.

Now, we apply the robust optimization approach such that the constraint functions are satisfied for all uncertainties $\mathcal{V}_j,~j\in J$, and we replace $(P_{u,v})$ by the single parameterized MOP
$$
(P_{u}):~~~\min (f_1(x)-\langle u_1,x\rangle, f_2(x)-\langle u_2,x\rangle,\ldots ,f_p(x)-\langle u_p,x\rangle)^T~~s.t.~~x\in \Omega,
$$
where
\begin{equation}\label{rfs}
\Omega :=\{x\in X : ~g_j(x,v_j)\leq 0,~~\forall v_j\in\mathcal{V}_j,~~j=1,2,\ldots,q \}.
\end{equation}
In fact, an element $x\in X$ is a feasible solution of $(P_u)$, known as a robust feasible solution,
if it is a feasible solution of $(P_{u,v})$ for all scenarios.
The feasible set $\Omega$ is called the robust feasible set.
Shortly we denote the objective function of $(P_{u})$ as $f(x,u):=(f_1(x,u_1),f_2(x,u_2),\ldots,f_p(x,u_p))^T$ where
$f_i(x,u_i):=f_i(x)-\langle u_i,x\rangle,~i\in I$. 
Furthermore, we consider the uncertain multi-objective optimization problem $(P_\mathcal{U}):=(P_u : u\in \mathcal{U})$ as
a family of parameterized Problems $(P_u)$.
Now, we are ready to define robust efficient solutions for $(P_\mathcal{U})$.

There are various robust solution notions in the literature; see \cite{ehr-m,fli,geo,ide-r}.
In the current work, the highly robustness, introduced by Ide and Sch\"{o}bel \cite{ide-r}, is investigated. This notion refers to the feasible solutions which are efficient in every scenario.

\begin{definition}\cite{ide-r}\label{hrob}
A vector $\bar{x}\in \Omega$ is called a highly robust efficient (resp. highly robust weakly/strictly efficient) solution to $(P_\mathcal{U})$ if for any $u\in\mathcal{U}$, the vector $\bar{x}$ is an efficient
(resp. a weakly/strictly efficient) solution of $(P_u)$.
\end{definition}

\begin{example}\label{exhrob}
Consider the parameterized MOP
\begin{equation}\label{prbexhrob}\begin{array}{l}
\min\, (x - u_1 x, x^2 - u_2 x)~~~s.t.~~-1\leq x\leq 1,
\end{array}
\end{equation}
with the uncertainty set $\mathcal{U}:=[-0.5,0.5]\times [0, 1]$. We show that $\bar{x}:=0$ is a highly
robust efficient solution of $(P_\mathcal{U})$. If not, there exist some $u=(u_1,u_2)\in \mathcal{U}$ and
$x\in \mathbb{R}$ such that $x\neq 0$ and
$(x - u_1 x, x^2 - u_2 x)\leqq (0,0)$, and $(x - u_1 x, x^2 - u_2 x)\neq (0,0).$
This results in $u_1\geq 1$ when $x>0$, and $u_2<0$ when $x<0$. Both cases
contradict $(u_1,u_2)\in [-0.5,0.5]\times [0,1]$.
\end{example}

\begin{proposition}\label{rem1}
If $\mathcal{U}_i,~i\in I$, are open sets in Banach space $X^*$ with weak$^*$ topology, then $\bar{x}\in \Omega$ is a highly robust efficient solution if and only if it is highly robust weakly efficient if and only if it is highly robust strictly efficient.
\end{proposition}
\begin{proof}
Assume that $\bar{x}\in\Omega$ is a highly robust weakly efficient solution of $(P_\mathcal{U})$ while it is not highly robust strictly efficient.
Then, there exist $\hat{x}\in \Omega\setminus \{\bar{x}\}$ and $\hat{u}\in\mathcal{U}$ such that
$f(\hat{x},\hat{u})\leqq f(\bar{x},\hat{u})$.
Consider the subspace $Y:=\{\alpha(\hat{x}-\bar{x}) : \alpha \in \Bbb R\}$ of $X$ and the continuous linear functional
$h:Y\longrightarrow \Bbb R$ defined as $h(\alpha(\hat{x}-\bar{x})):=\alpha\Vert\hat{x}-\bar{x}\Vert$.
Then, by a consequence of Hahn-Banach theorem \cite[Corollary 6.5]{con}, there exists a continuous linear functional $x^*\in X^*$ such
that $\langle x^*,\hat{x}-\bar{x}\rangle=\Vert\hat{x}-\bar{x}\Vert > 0$.
Obviously, $\hat{u}_i+\frac{1}{\nu}x^*\overset{w^*}{\longrightarrow} \hat{u}_i,~i\in I$, as $\nu\to\infty$. So, since
$\mathcal{U}_i,~i\in I$, are open sets, there exists $\nu\in \Bbb N$ such that
$\bar{u}_i:=\hat{u}_i+\frac{1}{\nu}x^*\in\mathcal{U}_i$ and
$\langle \hat{u}_i,\hat{x}-\bar{x}\rangle<\langle \bar{u}_i,\hat{x}-\bar{x}\rangle$, for any $i\in I$.
This implies $f(\hat{x},\bar{u})<f(\bar{x},\bar{u})$ which contradicts highly robust weak efficiency of $\bar{x}$. Therefore, highly robust weak efficiency implies highly robust strict efficiency. The other parts of the proposition are derived directly from Definition \ref{hrob}.
\end{proof}


In spite of the objective-wise uncertainty, the perturbation elements added to the objective functions ($u_i$'s) in Problem $(P_u)$ are not
necessarily equal to each other. However, in the following we show that if $\mathcal{U}_1=\mathcal{U}_2 =\ldots=\mathcal{U}_p$, then one may select
$u_i$ parameters equal to each other.
\begin{theorem}\label{h-h}
Let $\bar{\mathcal{U}}=\mathcal{U}_1=\mathcal{U}_2 =\ldots=\mathcal{U}_p$.
The vector $\bar{x}\in\Omega$ is a highly robust (weakly/strictly) efficient solution of $(P_{\mathcal{U}})$ if and only if
$\bar{x}$ is a highly robust (weakly/strictly) efficient solution of $(\bar{P}_{\bar{\mathcal{U}}}):= (\bar{P}_{u} : u\in\bar{\mathcal{U}})$
where
$$\begin{array}{l}
(\bar{P}_u):~~~\min\, (f_1(x)-\langle u,x\rangle, f_2(x)-\langle u,x\rangle,\ldots, f_p(x)-\langle u,x\rangle)^T~~s.t.~~x\in \Omega.
\end{array}$$
\end{theorem}
\begin{proof}
The proof of ``only if" part is straightforward.
To prove the converse, assume that $\bar{x}$ is a highly robust efficient solution of $(\bar{P}_{\bar{\mathcal{U}}})$,
while it is not a highly robust efficient solution of $(P_{\mathcal{U}})$.
Then, there exist $\hat{x}\in \Omega$ and $\hat{u}\in \mathcal{U}$ such that 
$$\begin{array}{l}
f_i(\hat{x}) - \langle \hat{u}_i,\hat{x}\rangle \leq f_i(\bar{x})- \langle \hat{u}_i,\bar{x}\rangle, ~\text{for~all}~i\in I,\vspace*{1mm}\\
f_i(\hat{x}) - \langle \hat{u}_i,\hat{x}\rangle < f_i(\bar{x})- \langle \hat{u}_i,\bar{x}\rangle, ~\text{for~some}~i\in I.
\end{array}$$
Set $\langle\bar{u},\hat{x}-\bar{x}\rangle =\max_{i\in I} \langle \hat{u}_i , \hat{x}-\bar{x}\rangle.$
We get $\bar{u}\in \bar{\mathcal{U}}$ and
$$\begin{array}{l}
f_i(\hat{x}) - \langle \bar{u},\hat{x}\rangle \leq f_i(\bar{x})- \langle \bar{u},\bar{x}\rangle, ~\text{for~all}~i\in I,\vspace*{1mm}\\
f_i(\hat{x}) - \langle \bar{u},\hat{x}\rangle < f_i(\bar{x})- \langle \bar{u},\bar{x}\rangle, ~\text{for~some}~i\in I,
\end{array}$$
which contradicts highly robust efficiency of $\bar{x}$ for $(\bar{P}_{\bar{\mathcal{U}}})$.
The proof for highly robust weak/strict efficiency is similar.
\end{proof}

Note that, due to Definition \ref{hrob}, a highly robust (weakly/strictly) efficient solution of $(P_\mathcal{U})$
may not be a (weakly/strictly) efficient solution of (\ref{prb}) when $0\notin \mathcal{U}_i,~i\in I$.
In the following proposition, we provide a condition which ensures it. 

\begin{proposition}\label{hrwe-se}
Let $\bar{x}\in\Omega$ be a highly robust weakly efficient solution of $(P_\mathcal{U})$.
Then, $\bar{x}$ is a strictly efficient solution of $(\ref{prb})$ if
\begin{equation}\label{eq-hrwe-se}
\forall x\in\Omega\setminus \{\bar{x}\},~\exists u\in\mathcal{U}; ~\langle u_i,x-\bar{x}\rangle>0,~\forall i\in I.
\end{equation}
\end{proposition}
\begin{proof}
Assume on the contrary that $\bar{x}$ is not a strictly efficient solution of (\ref{prb}).
Then, there exists $\hat{x}\in\Omega\setminus \{\bar{x}\}$ such that $f_i(\hat{x})\leq f_i(\bar{x})$ for each $i\in I.$
Also, by (\ref{eq-hrwe-se}), we get $\langle \hat{u}_i,\hat{x}-\bar{x}\rangle > 0,~i\in I$, for some $\hat{u}\in \mathcal{U}$.
So, we have
$$f_i(\hat{x})-f_i(\bar{x})\leq 0< \langle \hat{u}_i,\hat{x}-\bar{x}\rangle,~~\forall i\in I,$$
which implies
$$f_i(\hat{x})-\langle \hat{u}_i,\hat{x}\rangle < f_i(\bar{x})-\langle \hat{u}_i,\bar{x}\rangle,~~\forall i\in I.$$
This contradicts highly robust weak efficiency and the proof is complete.
\end{proof}

In the continuation of this section, we explore the connections between highly robust efficiency and various notions found in the literature, including robust, proper, and isolated solutions.

Set-based relation is one of the most well-known tools to investigate robustness in multi-objective optimization.
In the following, we provide the definition of set-based robust efficient solutions, introduced by Ehrgott et al. \cite{ehr-m}.

\begin{definition}\label{srob}
The vector $\bar{x}\in \Omega$ is called a set-based robust efficient solution of $(P_\mathcal{U})$, if there exists no $x\in \Omega$ such that
$f_\mathcal{U}(x)\subseteq f_\mathcal{U}(\bar{x})-\mathbb{R}^p_\geq,$
where $f_\mathcal{U}(x)=\{f(x, u): u\in \mathcal{U}\}$.
\end{definition}

Minmax (worst-case) robust optimality is another outstanding notion for uncertain optimization.
This solution concept was initially introduced by Soyster \cite{soy}, and then it was extensively studied
by Ben-Tal and Nemirovski \cite{ben-ro} and Ben-Tal et al. \cite{ben-b}. Fliege and Werner \cite{fli},
Bokrantz and Fredriksson \cite{bok}, and Kuroiwa and Lee \cite{kur} developed this concept for multi-objective programming.

Consider the uncertain multi-objective optimization problem $(P_\mathcal{U})$.
For each $i\in I$, define $F_i^{\mathcal{U}_i}:X\longrightarrow \overline{\mathbb{R}}$ by
$F_i^{\mathcal{U}_i}(x) := \sup_{u_i\in\mathcal{U}_i} f_i(x, u_i),$
and $F^\mathcal{U} : X\longrightarrow \mathbb{R}^p$ by $F^\mathcal{U}(x) := (F_1^{\mathcal{U}_1}(x), \ldots , F_p^{\mathcal{U}_p}(x))^T$.

\begin{definition}\label{wrob}
The vector $\bar{x}\in \Omega$ is called a worst-case robust efficient solution of $(P_\mathcal{U})$, if it is an efficient solution of \begin{equation}\label{prbw}
\begin{array}{l}
\min\, F^{\mathcal{U}}(x)~~~s.t.~~~x \in \Omega.
\end{array}
\end{equation}
\end{definition}

Theorem \ref{h-sw} shows that highly robustness is sufficient for set-based and worst-case robustness.

\begin{theorem}\label{h-sw}
Let the uncertainty sets $\mathcal{U}_i,~i\in I$, be $w^*$-compact. If $\bar{x}\in\Omega$ is a highly robust efficient solution of $(P_\mathcal{U})$, then $\bar{x}$ is a set-based and worst-case robust efficient solution of $(P_\mathcal{U})$.
\end{theorem}
\begin{proof}
Invoking \cite[Theorem 4.3(b)]{rud}, there exists a unique $\Phi_{\bar{x}}\in X^{**}$ such that $\langle \Phi_{\bar{x}},x^*\rangle =\langle x^*,\bar{x}\rangle$ for each $x^*\in X^*$, where $X^{**}:=(X^*)^*$ is the second dual of $X$.
Now, since the uncertainty sets $\mathcal{U}_i,~i\in I$, are $w^*$-compact, the continuous functional $\Phi_{\bar{x}}$ attains its infimum on
$\mathcal{U}_i,~i\in I$; see \cite[Page 402]{rud}. So, there exists some $\bar{u}\in\mathcal{U}$ such that for each $i\in I$,
$\langle \bar{u}_i ,\bar{x}\rangle =\langle \Phi_{\bar{x}}, \bar{u}_i\rangle =\min_{u_i\in\mathcal{U}_i} \langle \Phi_{\bar{x}}, u_i\rangle
=\min_{u_i\in\mathcal{U}_i} \langle u_i ,\bar{x}\rangle.$
If $\bar{x}$ is not set-based robust efficient, then $f_\mathcal{U}(\hat{x})\subseteq f_\mathcal{U}(\bar{x})-\mathbb{R}^p_\geq$
for some $\hat{x}\in\Omega$.
So, $f(\hat{x},\bar{u}) \in f_\mathcal{U}(\hat{x})$ implies that there exists some $\hat{u}\in\mathcal{U}$ such that
$$\begin{array}{ll}
f(\hat{x},\bar{u}) &\leq f(\bar{x},\hat{u})=f(\bar{x})-\big(\langle \hat{u}_1,\bar{x}\rangle, \langle \hat{u}_2,\bar{x}\rangle,\ldots, \langle \hat{u}_p,\bar{x}\rangle\big)^T\\
&\leqq f(\bar{x})-\big(\langle \bar{u}_1,\bar{x}\rangle, \langle \bar{u}_2,\bar{x}\rangle,\ldots, \langle \bar{u}_p,\bar{x}\rangle\big)^T=f(\bar{x},\bar{u}).
\end{array}$$
This contradicts the highly robustness of $\bar{x}$.

If $\bar{x}$ is not a worst-case robust efficient solution, then there exists some $\hat{x}\in\Omega$ such
that $F^\mathcal{U}(\hat{x})\leq F^\mathcal{U}(\bar{x})$. On the other hand, for each $i\in I$,
$$\left\lbrace\begin{array}{l}
F_i^{\mathcal{U}_i}(\bar{x}) = \displaystyle\sup_{u_i\in \mathcal{U}_i}f_i(\bar{x},u_i) = f_i(\bar{x}) - \min_{u_i\in \mathcal{U}_i} \langle u_i ,\bar{x}\rangle = f_i(\bar{x}) - \langle \bar{u}_i ,\bar{x}\rangle = f_i(\bar{x},\bar{u}_i),\vspace*{1mm}\\
f_i(\hat{x},\bar{u}_i)\leq \displaystyle\sup_{u_i\in \mathcal{U}_i}f_i(\hat{x},u_i) = F_i^{\mathcal{U}_i}(\hat{x}).
\end{array}\right.$$
So, we get $f(\hat{x},\bar{u})\leq f(\bar{x},\bar{u})$, which contradicts the assumption.
\end{proof}

As an example, consider the parameterized Problem (\ref{prbexhrob}) with uncertainty set $\mathcal{U}=[-0.5,0.5]\times [0,1]$.
Due to Example \ref{exhrob} and Theorem \ref{h-sw}, $\bar{x}=0$ 
is a set-based and worst-case robust efficient solution.



Analogous to Theorem \ref{h-sw}, it can be proven that high robustness is sufficient for other robustness concepts defined by means of set relations. In the next sections, we derive some sufficient conditions for high robustness. As per the above theorem, these conditions are also sufficient for set-based and worst-case robustness, rendering the verification of these solutions operational.


We proceed in this section by highlighting an interesting aspect of highly robust solutions. It is demonstrated that in finite-dimensional spaces, under certain conditions, the set of Henig properly efficient solutions contains the set of highly robust efficient solutions. Subsequently, we extend this result to infinite-dimensional spaces under specific constraint qualifications.

\begin{theorem}\label{h-pf}
Assume that $X$ is finite-dimensional and $\bar{x}\in\Omega$ is a highly robust weakly efficient solution of $(P_\mathcal{U})$.
Let either (i) $\Omega$ be compact and $f_i$'s be locally Lipschitz at $\bar{x}$ or (ii) $\Omega$ be convex and $f_i$'s be locally convex at $\bar{x}$.
Then, $\bar{x}$ is a Henig properly efficient solution of $(\ref{prb})$ if the following condition is satisfied:
\begin{equation}\label{eq1-h-pf}
\forall d\in T,~\exists\, u\in \mathcal{U};~  \langle u_i,d\rangle > 0,~~\forall i\in I,
\end{equation}
where
$$T:=\{d\in X : ~\exists\{x_\nu\}_\nu\subseteq \Omega\setminus \{\bar{x}\};~\frac{x_\nu-\bar{x}}{\Vert x_\nu-\bar{x}\Vert}\longrightarrow d\}.$$
\end{theorem}
\begin{proof}
Due to Proposition \ref{hrwe-se}, $\bar{x}$ is an efficient solution of (\ref{prb}).
Set
$$\Gamma :=\{a\in \mathbb{R}^p :  ~\forall x \in\Omega\setminus\{\bar{x}\},~\exists u\in\mathcal{U};
~-\langle u_i,\frac{x-\bar{x}}{\Vert x-\bar{x}\Vert}\rangle<a_i,~\forall i\in I\}.$$
We get $\mathbb{R}^p_\geqq\subseteq\Gamma$ by (\ref{eq1-h-pf}). Due to \cite[Lemma 3.4]{hen-p} and \cite[Lemma 3.1]{hen-a}, there exists a sequence of
sets $\{K_\nu\}_\nu$ in $\mathbb{R}^p$ such that the following properties hold:
\begin{itemize}
\item[$\bullet$] The sets $K_\nu$'s are closed pointed convex cones;
\item[$\bullet$] $\mathbb{R}^p_\geqq\subseteq int\,K_\nu \cup \{0\}$ for each $\nu$;
\item[$\bullet$] $K_{\nu+1}\subseteq int\,K_\nu \cup \{0\}$ for each $\nu$;
\item[$\bullet$] Every bounded sequence $\{\alpha_\nu\}_\nu$ with $\alpha_\nu \in K_\nu$,
has a convergent subsequence with a limit in $\mathbb{R}^p_\geqq$.
\end{itemize}
Assume that $L_i$ is a Lipschitz modulus of $f_i$ on a neighbourhood of $\bar{x}$ and $L := \max_{i\in I} L_i$. We claim there exists some $\nu_0$ such that $\Big(K_\nu\cap cl\,\mathbb{B}(0;L)\Big)\subseteq\Gamma$ for each $\nu\geq \nu_0$.
If not, for each $\nu$, there exists $\alpha_\nu \in K_\nu\cap cl\,\mathbb{B}(0;L)$ with $\alpha_\nu \notin\Gamma$.
This means, for each $\nu$, there is $x_\nu \in\Omega\setminus \{\bar{x}\}$ such that
\begin{equation}\label{eq2-h-pf}
\forall u\in\mathcal{U},~\exists\,i_\nu\in I;~-\langle u_{i_\nu},\frac{x_\nu-\bar{x}}{\Vert x_\nu-\bar{x}\Vert}\rangle \geq \alpha_{i_\nu\nu}.
\end{equation}
Define $A^u_\nu:=\big\{i\in I:~-\langle u_{i},\frac{x_\nu-\bar{x}}{\Vert x_\nu-\bar{x}\Vert}\rangle \geq \alpha_{i\nu}\big\}$.
On account of (\ref{eq2-h-pf}), this set is nonempty.
Without loss of generality, by choosing an appropriate subsequence, $A^u_\nu$ is a constant set, $A^u$, for all $\nu$ indices.
Also, taking the properties of $K_\nu$ cones into account, $\alpha_\nu$ converges to some $\bar{\alpha}\in\mathbb{R}^p_\geqq\subseteq\Gamma$ with $\Vert\bar{\alpha}\Vert\leq L$.
Now, since the sequence $\frac{x_\nu-\bar{x}}{\Vert x_\nu-\bar{x}\Vert}$ is bounded, by choosing an appropriate subsequence if necessary,
$\frac{x_\nu-\bar{x}}{\Vert x_\nu-\bar{x}\Vert}\longrightarrow d$ for some $d\in X$.
So, $d\in T$.
On the other hand, by $\nu\longrightarrow\infty$ in (\ref{eq2-h-pf}), we get
$$\forall u\in\mathcal{U}:~~-\langle u_i,d\rangle\geq \bar{\alpha}_i\geq 0,~\forall i\in A^u,$$
which contradicts (\ref{eq1-h-pf}).
Hence, there exists  some $\nu_0\in\mathbb{N}$ such that $\Big( K_\nu\cap cl\,\mathbb{B}(0;L)\Big)\subseteq \Gamma$ for each $\nu\geq \nu_0$.
We show that $\bar{x}$ is an efficient solution of (\ref{prb}) with respect to the cone $K_\nu$ for some $\nu\geq \nu_0$.
If not, for each $\nu\geq \nu_0$ there exists $y_\nu \in \Omega\setminus \{\bar{x}\}$ such that
\begin{equation}\label{eq3-h-pf}
f(y_\nu)-f(\bar{x})\in -K_\nu \setminus \{0\}.
\end{equation}
In this situation, if there is some $\bar\nu \geq \nu_0$ such that $\Vert f(y_{\bar{\nu}})-f(\bar{x})\Vert \leq L\Vert y_{\bar{\nu}}-\bar{x}\Vert$,
as $\Big( K_{\bar{\nu}}\cap cl\,\mathbb{B}(0;L)\Big)\subseteq \Gamma$, we obtain that for some $u\in \mathcal{U},$
$$\begin{array}{l}
~~~-\langle u_i,\frac{y_{\bar{\nu}}-\bar{x}}{\Vert y_{\bar{\nu}}-\bar{x}\Vert}\rangle <
\frac{f_i(\bar{x})-f_i(y_{\bar{\nu}})}{\Vert y_{\bar{\nu}}-\bar{x}\Vert},~\forall i\in I\Longrightarrow -\langle u_i,y_{\bar{\nu}}-\bar{x}\rangle < f_i(\bar{x})-f_i(y_{\bar{\nu}}),~~\forall i\in I\vspace*{1mm}\\
\Longrightarrow f_i(y_{\bar{\nu}})-\langle u_i,y_{\bar{\nu}}\rangle < f_i(\bar{x})-\langle u_i,\bar{x}\rangle,~~\forall i\in I\Longrightarrow f(y_{\bar{\nu}},u)<f(\bar{x},u).
\end{array}
$$
This contradicts highly robust weak efficiency of $\bar{x}$.
Otherwise, let $\Vert f(y_\nu)-f(\bar{x})\Vert > L\Vert y_\nu-\bar{x}\Vert$ for each $\nu\geq \nu_0$.
Now, we continue the proof by considering two assumed cases:
first, compactness of $\Omega$ and locally Lipschitzness of $f_i$'s; and second, convexity of $\Omega$ and local convexity of $f_i$'s.\\
In the both cases, $f_i$'s are locally Lipschitz at $\bar{x}$ (note that locally convexity leads in locally Lipschitzness for real-valued functions).
So, $y_\nu$ has not any convergent subsequence to $\bar{x}$.
In the first case, as $\Omega$ is compact and $f_i$'s are locally Lipschitz at $\bar{x}$,
without loss of generality, by working with an appropriate subsequence, $y_\nu$ converges to some $\hat{y}\in\Omega$.
Then, $f(y_\nu) \longrightarrow f(\hat{y})$, $\hat{y} \neq \bar{x}$, and $f(\hat{y}) \neq f(\bar{x})$.
Furthermore, taking the properties of $K_\nu$ cones into account, $f(\hat{y}) \leq f(\bar{x})$ and this contradicts efficiency of $\bar{x}$.
Therefore, $\bar{x}$ is an efficient solution of (\ref{prb}) with respect to cone $K_\nu$ for some $\nu$.
Thus, $\bar{x}$ is a Henig properly efficient solution of (\ref{prb}).

Consider the second case: $\Omega$ is convex and $f_i$'s are locally convex at $\bar{x}$.
Since $y_\nu$ has not any convergent subsequence to $\bar{x}$, without loss of generality, there exists some scalar $r>0$ such that
$\Vert y_\nu -\bar{x}\Vert >r$.
Set $$z_\nu := \frac{r}{\nu \Vert y_\nu -\bar{x}\Vert} y_\nu +(1-\frac{r}{\nu \Vert y_\nu -\bar{x}\Vert}) \bar{x}= \frac{r(y_\nu -\bar{x})}{\nu \Vert y_\nu -\bar{x}\Vert}+\bar{x} \in \Omega\setminus \{\bar{x}\}.$$
As $z_\nu \longrightarrow \bar{x}$, from locally convexity of $f_i$'s we get
$$f(z_\nu)-\frac{r}{\nu \Vert y_\nu -\bar{x}\Vert} f(y_\nu) - (1-\frac{r}{\nu \Vert y_\nu -\bar{x}\Vert}) f(\bar{x})
\in \mathbb{R}^p_\geqq \subseteq int\, K_\nu \cup \{0\},$$
and by (\ref{eq3-h-pf}), $f(z_\nu) -f(\bar{x})\in K_\nu \setminus \{0\}.$
On the other hand, thanks to locally Lipschitzness of $f_i$'s,
$\Vert f(z_\nu)-f(\bar{x})\Vert \leq L\Vert z_\nu-\bar{x}\Vert$ for sufficiently large $\nu$.
Therefore, the desired result is obtained from the previous part by working with $z_\nu$ instead of $y_\nu.$ The proof is complete.
\end{proof}


According to Theorem \ref{h-pf}, one can reject highly robustness of some solutions of special-structured problems easily, as checking Henig proper efficiency could be easy geometrically. For instance, in the bi-objective problem
$$\begin{array}{l}
\min\,(x_1,x_2),~~s.t.~~x_1^2+x_2^2\leq 1,
\end{array}$$
it is easy to show that $(0,-1)$ is not a Henig properly efficient.
So, by Theorem \ref{h-pf}, for each uncertainty set $\mathcal{U}$ that satisfies (\ref{eq1-h-pf}),
this solution is not highly robust.

Example \ref{exhrob} illustrated that $\bar{x}=0$ is a highly robust efficient solution of (\ref{prbexhrob}).
It is not difficult to see that condition (\ref{eq1-h-pf}) is not satisfied for this problem.
On the other hand, $\bar{x}=0$ is not a properly efficient solution of this problem for $u=(0,0)$.
Therefore, condition (\ref{eq1-h-pf}) in Theorem \ref{h-pf} cannot be omitted.

Now, we provide a Constraint Qualification (CQ) to obtain a result similar to Theorem \ref{h-pf} in infinite-dimensional spaces.
We say CQ1 is satisfied if
$$0\notin T_w:=\{d\in X : ~\exists\{x_\nu\}_\nu\subseteq \Omega\setminus \{\bar{x}\};~\frac{x_\nu-\bar{x}}{\Vert x_\nu-\bar{x}\Vert}\overset{w}{\longrightarrow} d\}.$$


\begin{theorem}\label{h-pi}
Assume that $X$ is a reflexive Banach space, $\bar{x}\in\Omega$ is a highly robust weakly efficient solution of $(P_\mathcal{U})$, and
CQ1 is satisfied.
If either (i) $\Omega$ is compact and $f_i$'s are locally Lipschitz at $\bar{x}$ or (ii) $\Omega$ is convex and $f_i$'s are locally convex at $\bar{x}$, then $\bar{x}$ is a Henig properly efficient solution of $(\ref{prb})$ if the following condition is satisfied:
\begin{equation}\label{eq-h-pi}
\forall d\in T_w,~\exists\, u\in \mathcal{U};~  \langle u_i,d\rangle > 0,~~\forall i\in I.
\end{equation}
\end{theorem}
\begin{proof}
The proof is similar to that of Theorem \ref{h-pf} by using the fact that every bounded sequence in a reflexive Banach space
has a weakly convergent subsequence as a corollary of Eberlein-$\check{\text{S}}$mulian Theorem \cite[Theorem 13.1]{con} together with CQ1.
\end{proof}


We close the section by providing some connections between isolated efficiency and highly robust efficiency.
Theorem \ref{i-h} shows that isolated efficiency is sufficient for highly robustness under boundedness  of the uncertainty sets.

\begin{theorem}\label{i-h}
Assume that there exists scalar $L>0$ such that $\Vert u_i\Vert < L$ for each $u_i\in\mathcal{U}_i,~i \in I$.
Then, each isolated efficient solution of $(\ref{prb})$ with constant $L$ is a highly robust strictly efficient solution of $(P_\mathcal{U})$.
\end{theorem}
\begin{proof}
Let $\bar{x}\in\Omega$ be an isolated efficient solution of (\ref{prb}) with constant $L$ and it is not highly robust strictly efficient.
Then, there exist $\hat{x}\in\Omega\setminus \{\bar{x}\}$ and $\hat{u}\in\mathcal{U}$ such that
$f(\hat{x},\hat{u})\leqq f(\bar{x},\hat{u})$. This
implies
$$\begin{array}{l}
f_i(\hat{x}) - f_i(\bar{x}) \leq \langle\hat{u}_i, \hat{x} - \bar{x}\rangle,~~\forall i\in I.
\end{array}$$
Thus,
$$\displaystyle\max_{i\in I} \big\{f_i(\hat{x}) - f_i(\bar{x})\big\} \leq \displaystyle\max_{i\in I} \big\{\langle\hat{u}_i, \hat{x} - \bar{x}\rangle\big\}\vspace*{2mm}\\
\leq \displaystyle\max_{i\in I} \big\{\Vert \hat{u}_i\Vert \Vert \hat{x} - \bar{x}\Vert \big\}\vspace*{2mm}\\
< L\Vert \hat{x} - \bar{x}\Vert.$$
This makes a contradiction with isolated efficiency of $\bar{x}$ and completes the proof.
\end{proof}

Theorem \ref{h-i} provides a converse version of Theorem \ref{i-h}.

\begin{theorem}\label{h-i}
Let $\bar{x}\in\Omega$ be a highly robust weakly efficient solution of $(P_\mathcal{U})$.
Let there exist some scalar $L>0$ such that $\mathcal{A}_L(x)\cap \mathcal{U}_i\neq\emptyset$ for each $x\in\Omega$ and for
each $i\in I$, where $\mathcal{A}_L(x)$ is the set of all vectors $x^*\in X^*$
satisfying $\langle x^*,x-\bar{x}\rangle = L\Vert x-\bar{x}\Vert$.
Then, $\bar{x}$ is an isolated efficient solution of $(\ref{prb})$ with constant $L$.
\end{theorem}
\begin{proof}
Consider $x\in\Omega\setminus \{\bar{x}\}$ and $k > 0$ arbitrary and constant hereafter.
Define the one-dimensional subspace $M$ of $X$ as $M:=\big\{\alpha (x - \bar{x}) :~ \alpha\in \mathbb{R}\big\}$ and
the linear functional $h:M\longrightarrow \mathbb{R}$ as $h(\alpha(x - \bar{x})) := k\alpha\Vert x - \bar{x}\Vert$.
Obviously, $h$ is continuous on $M$ and so, by a consequence of Hahn-Banach theorem \cite[Corollary 6.5]{con},
there exists a continuous linear functional $x^*\in X^*$ such that $\langle x^*, x - \bar{x}\rangle =k\Vert x - \bar{x}\Vert$.
Therefore, $\mathcal{A}_k(x)\neq\emptyset$, for each $x\in\Omega\setminus \{\bar{x}\}$ and $k>0$.
So, noting the assumption $\mathcal{A}_L(x)\cap \mathcal{U}_i\neq\emptyset$, for each $x\neq \bar x$, the vector $\bar{x}$ is an efficient solution of (\ref{prb}) due to Proposition \ref{hrwe-se}.
Now, to prove isolated efficiency of $\bar{x}$ with constant $L$. Assume on the contrary that, there exists some $\hat{x}\in\Omega\setminus \{\bar{x}\}$ such that
$$\displaystyle\max_{i\in I} \big\{f_i(\hat{x})-f_i(\bar{x})\big\} < L\Vert \hat{x} - \bar{x}\Vert.$$
Select $x^*_i \in \mathcal{A}_L(\hat{x})\cap \mathcal{U}_i,~i\in I,$ and set $\hat{u}:=(x^*_1, x^*_2,\ldots, x^*_p)$. We get
$$\begin{array}{ll}
&f_i(\hat{x}) - f_i(\bar{x}) < L\Vert \hat{x} - \bar{x}\Vert,~~\forall i \in I\Longrightarrow f_i(\hat{x}) - f_i(\bar{x}) < \langle x^*_i , \hat{x} - \bar{x}\rangle,~~\forall i \in I\vspace*{1mm}\\
\Longrightarrow &f_i(\hat{x}) - \langle x^*_i , \hat{x}\rangle < f_i(\bar{x}) - \langle x^*_i , \bar{x}\rangle,~~\forall i \in I
\Longrightarrow f(\hat{x}, \hat{u}) < f(\bar{x}, \hat{u}).
\end{array}$$
This contradicts highly robust weak efficiency of $\bar{x}$.
So, $\bar{x}$ is an isolated efficient solution of (\ref{prb}) with constant $L$ and the proof is complete.
\end{proof}


\section{Characterization of Highly Robustness}\label{section4}

This section is devoted to providing some optimality conditions to characterizing highly robust solutions.
In details, we exploit variational analysis tools, including the limiting subdifferential, a nonsmooth version of Fermat's rule, the mean
value theorem, and weak contingent cone, to establish necessary and sufficient conditions for
(local) highly robust solutions of $(P_\mathcal{U})$.

A Banach space $X$ is called weakly compactly generated (WCG) if
there exists a weakly compact set $K \subseteq X$ such that $X = cl (span\,K)$.
WCG Banach spaces build a large class of Banach spaces, including reflexive Banach spaces and separable Banach spaces; see \cite{mor-b1}.

Theorem \ref{nec1} provides necessary conditions for local highly robustness.

\begin{theorem}\label{nec1}
Let $X$ be a WCG Asplund space.
Let $\bar{x}\in\Omega$ be a local highly robust weakly efficient solution of $(P_\mathcal{U})$ and the functions $f_i,~i\in I$, be locally Lipschitz at $\bar{x}$.
Then, there exist no $d\in T_w(\bar{x};\Omega)$ and $u\in\mathcal{U}$ such that
$$\langle x^*-u_i, d\rangle <0,~~\forall x^*\in \partial f_i(\bar{x}),~\forall i\in I.$$
\end{theorem}
\begin{proof}
Let $f_i,~i\in I$, be locally Lipschitz at $\bar{x}$ with constant $L_i > 0$.
By indirect proof, assume there exists $d\in T_w(\bar{x};\Omega)$ and $u^d\in\mathcal{U}$ such that
\begin{equation}\label{eq1-nec1}
\langle x^*-u_i^d, d\rangle <0 ,~~\forall x^*\in \partial f_i(\bar{x}),~\forall i\in I.
\end{equation}
As $d\in T_w(\bar{x};\Omega)$, there are two sequences $\{x_\nu\}_\nu \subseteq \Omega$
and $\{t_\nu\}_\nu \subseteq \mathbb{R}_>$ satisfying
$x_\nu \longrightarrow\bar{x}$ and $\frac{x_\nu - \bar{x}}{t_\nu}\overset{w}{\longrightarrow} d.$
By applying the mean value theorem \cite[Corollary 3.51]{mor-b1} to the function $f_i$, $i\in I$, there are two vectors $y_{i\nu}\in (\bar{x}, x_\nu)$ and $x^*_{i\nu}\in \partial f_i(y_{i\nu})$ such that
\begin{equation}\label{eq2-nec1}
f_i(x_\nu) - f_i(\bar{x}) \leq\langle x^*_{i\nu}, x_\nu -\bar{x}\rangle.
\end{equation}
Clearly, $y_{i\nu}\longrightarrow \bar{x}$ as $\nu\to\infty$ , for each $i\in I$.
Furthermore, by (\ref{eqs}), $\Vert x^*_{i\nu}\Vert\leq L_i$ for each $\nu\in\mathbb{N}$. Invoking the weak$^*$ sequential compactness property of the bounded sets in Asplund spaces,
by working with a subsequence if necessary, $x^*_{i\nu} \overset{w}{\longrightarrow} \bar{x}^*_i$ for some $\bar{x}^*_i \in X^*$.
On the other side, on account of weak$^*$ closedness property of the set-valued function $\partial f_i:X\rightrightarrows X^*$
in WCG Asplund spaces, as a result of \cite[Theorem 3.60]{mor-b1},
$\bar{x}^*_i\in \partial f_i(\bar{x})$.
Therefore, it stems from (\ref{eq1-nec1}) that $\langle \bar{x}^*_i, d\rangle < \langle u^d_i, d\rangle$ for each $i\in I.$
This gives us that for sufficiently large $\nu$,
\begin{equation}\label{eq3-nec1}
\begin{array}{ll}
 &\langle \bar{x}^*_i, x_\nu - \bar{x}\rangle < \langle u^d_i, x_\nu - \bar{x}\rangle ,~~\forall i\in I.
\end{array}
\end{equation}
Combining (\ref{eq2-nec1}) with (\ref{eq3-nec1}), we get
$f_i(x_\nu) - \langle u^d_i, x_\nu \rangle < f_i(\bar{x}) -\langle u^d_i, \bar{x}\rangle,$ for all $i\in I$, leading in $f(x_\nu, u^d) < f(\bar{x}, u^d).$
This contradicts with local highly robust weak efficiency of $\bar{x}$ and completes the proof.
\end{proof}

Let us illustrate the necessary condition obtained in Theorem \ref{nec1} by the following example.

\begin{example}\label{ex1-nec1}
Consider an uncertain multi-objective optimization problem $(P_\mathcal{U})$
with $f_i:\mathbb{R}^2\longrightarrow\mathbb{R},~i=1, 2$, given by
$$f_1(x_1, x_2) := \vert x_1 + 1\vert + x_2,~~~~~ f_2(x_1, x_2) := x_1 + \vert x_2 + 1\vert,$$
$\Omega := [-1, 1]\times [-1, 1]\subseteq \mathbb{R}^2$, and
$\mathcal{U}_1 = \mathcal{U}_2 := [-1, 0]\times [-1, 0] \subseteq \mathbb{R}^2$.
It can be seen that $\bar{x}=(-1, -1)$ is a highly robust efficient solution of $(P_\mathcal{U})$.
Furthermore, here, $T_w(\bar{x};\Omega) = \mathbb{R}^2_\geqq$,
$\partial f_1(\bar{x}) = [-1, 1]\times \{1\},$ and $\partial f_2(\bar{x}) = \{1\}\times [-1, 1].$
It can be easily seen that the necessary condition derived in Theorem \ref{nec1} holds.
\end{example}

To obtain a version of Theorem \ref{nec1} in general real Banach spaces, we use the Clarke's generalized gradient instead of Mordukhovich subdifferential.

\begin{theorem}\label{nec2}
Let $X$ be a real Banach space.
If $\bar{x}\in\Omega$ is a local highly robust weakly efficient solution of $(P_\mathcal{U})$
and the functions $f_i,~i\in I$, are locally Lipschitz at $\bar{x}$, then there exist no $d\in T_w(\bar{x};\Omega)$ and $u\in\mathcal{U}$ such that
$$\langle x^*-u_i, d\rangle <0,~~\forall x^*\in \partial^C f_i(\bar{x}),~\forall i\in I.$$
\end{theorem}
\begin{proof}
It is derived from Theorem \ref{nec1} due to Eq. (\ref{eqcm}).
\end{proof}

In the following example, we show that the converse of Theorem \ref{nec1} may
not be valid, in general, even in a smooth setting.

\begin{example}\label{ex2-nec1}
Let $f_i:\mathbb{R}\longrightarrow\mathbb{R},~i=1, 2$, be given as $f_1(x):= x$ and $f_2(x):= -\sqrt[3]{x}$. Furthermore, let $\Omega := (-\infty, 1]$ be the feasible set, and $\mathcal{U} = [-1, 1] \times [-1, 1]$ be the uncertainty set.
Then, $(P_\mathcal{U})$ is formulated by
$$\begin{array}{l}
\min\, (x - u_1x, -\sqrt[3]{x} - u_2 x)~~s.t.~~x \leq 1,
\end{array}$$
where $u_1, u_2 \in [-1, 1]$. Consider $\bar{x} := -1$. We get $T_w(\bar{x};\Omega) = \mathbb{R}$,
$\partial f_1(\bar{x}) = \{1\}$, and $\partial f_2(\bar{x}) = \{\frac{-1}{3}\}$.
The necessary condition provided in Theorem \ref{nec1} is fulfilled while $\bar{x}$ is not a local highly
robust weakly efficient solution of $(P_\mathcal{U})$ (Set $\hat{x} := -1$, $\hat{u}_1 := 0$, and $\hat{u}_2 := -1$).
\end{example}

As shown in Example \ref{ex2-nec1}, the converse of Theorem \ref{nec1} may not be true in general.
Therefore, for the purpose of obtaining sufficient conditions, we define a generalized convexity notion as follows.

\begin{definition}
We say that $f$ is generalized convex (resp. strictly generalized convex) at $\bar{x}\in \Omega$, on the set $\Omega$, if for each $x \in \Omega$
and each $u\in \mathcal{U}$ there exists $d\in X$ such that
$$\begin{array}{ll}
\langle x^*_i, d\rangle \leq f_i(x,u_i) - f_i(\bar{x},u_i),~~~~\forall x^*_i\in \partial_x f_i(\bar{x},u_i),~\forall i\in I,\vspace{1mm}\\
\big(\text{resp.}~~\langle x^*_i, d\rangle < f_i(x,u_i) - f_i(\bar{x},u_i),~~~~\forall x^*_i\in \partial_x f_i(\bar{x},u_i),~\forall i\in I\big).
\end{array}$$
Furthermore, if the conditional parts occur for each $x\in \Omega\cap\mathcal{B}$ for some neighbourhood $\mathcal{B}$
of $\bar{x}$ instead of $\Omega$, then $f$ is said to be (strictly) generalized \textit{locally} convex at $\bar{x}$.
\end{definition}

About the above definition, two points are worth saying. (1) If $f_i$'s are
convex on $\Omega$ (resp. locally convex at $\bar{x}$ on $\Omega$), then $f$ is generalized convex at
$\bar{x}$ on $\Omega$ (resp. generalized locally convex at $\bar{x}$ on $\Omega$), due to \cite[Theorem 1.93]{mor-b1}
by setting $d:=x-\bar{x}$ for each $x\in \Omega$ (resp. $x\in \Omega\cap\mathcal{B}$). (2) The function
$$f(x,u):=\left\lbrace\begin{array}{ll}
2\vert x\vert +x^2\sin\frac{1}{x}+ux,~~~&x\neq 0,\vspace*{0.5mm}\\
0,~~~&x=0,
\end{array}\right.$$
is generalized convex at $\bar{x}:= 0$ on $\Omega := [-1, 1]$ with $u\in \mathcal{U}:=[0,1]$
(We have, $0\leq h(x)$ for each $x\in \Omega$ with $d=0$)
while it is not even locally convex at this point
(Consider $u:=0$, $x := \frac{1}{2k\pi +\frac{3\pi}{2}}$, for $k\in\mathbb{N},$ $y := 0$, and $\lambda :=\frac{1}{2}$).
So, the class of generalized (locally) convex functions at a given point on a set is larger than that of (locally) convex
functions.

Now, we are ready to obtain sufficient conditions for highly robustness.
Theorem \ref{suf1} provides a converse version of Theorem \ref{nec1}.

\begin{theorem}\label{suf1}
Let $\bar{x}\in\Omega$.
Assume that there exist no $d\in X$ and $u\in \mathcal{U}$ such that
$$\langle x^*-u_i,d\rangle<0,~~~\forall x^*\in \partial f_i(\bar{x}),~~\forall i\in I.$$
\begin{itemize}
\item[i)] If $f$ is generalized locally convex at $\bar{x}$ on $\Omega$, then
$\bar{x}$ is a local highly robust weakly efficient solution of $(P_\mathcal{U})$.
\item[ii)] If $f$ is strictly generalized locally convex at $\bar{x}$ on $\Omega$, then
$\bar{x}$ is a local highly robust strictly efficient solution of $(P_\mathcal{U})$.
\end{itemize}
\end{theorem}
\begin{proof}
(i) By generalized local convexity assumption, there exists some neighbourhoods $\mathcal{B}$ of $\bar{x}$ such that for each $x \in \Omega\cap\mathcal{B}$ and each $u\in \mathcal{U}$ there exists $d\in X$ such that
$$
\langle z^*, d\rangle \leq f_i(x,u_i) - f_i(\bar{x},u_i),~~~\forall z^*\in \partial_x f_i(\bar{x},u_i),~\forall i\in I.
$$
By applying a subdifferential sum rule \cite[Proposition 1.107]{cla-f}, we get
\begin{equation}\label{eq1-suf1}
\langle x^*-u_i, d\rangle \leq f_i(x,u_i) - f_i(\bar{x},u_i),~~~\forall x^*\in \partial f_i(\bar{x}),~\forall i\in I.
\end{equation}
If there exist $\hat{x}\in\Omega\cap \mathcal{B}\setminus \{\bar{x}\}$ and $\hat{u}\in\mathcal{U}$ such that $f(\hat{x}, \hat{u}) < f(\bar{x}, \hat{u})$, then, by (\ref{eq1-suf1}), there exists $d\in X$ such that
$$\langle x^*-\hat{u}_i, d\rangle \leq f_i(\hat{x},\hat{u}_i) - f_i(\bar{x},\hat{u}_i)<0,~~~\forall x^*\in \partial f_i(\bar{x}),
~\forall i\in I$$
The last inequality makes a contradiction. The proof of part (i) is complete.\\
(ii) It is proved similar to (i).
\end{proof}


Analogous to Theorem \ref{suf1}, sufficient conditions for \textit{global} highly robustness can be derived.



In the continuation of the current section, we obtain necessary and sufficient conditions for highly robust solutions of $(P_\mathcal{U})$
when the robust feasible set is as described in (\ref{rfs}). We need to define some basic concepts.

Consider $j\in J$. Define a pointwise supremum mapping $G_j:X\longrightarrow\overline{\mathbb{R}},~j\in J,$ as
$G_j(x):=\sup_{v_j\in\mathcal{V}_j} g_j(x,v_j),~x\in X.$
Then, $\Omega =\big\{x\in X :~ G_j(x)\leq 0,~\forall j\in J\big\}$.
Given $\varepsilon \geq 0$ and $\bar{x}\in X$, we define the perturbed set of active indices in $\mathcal{V}_j$ at $\bar{x}\in X$
(see \cite{mor-c}) by
$$\mathcal{V}^\varepsilon_j(\bar{x}):=\big\{v_j\in \mathcal{V}_j \bigm| ~g_j(\bar{x},v_j)\geq G_j(\bar{x})-\varepsilon\big\},$$
and in particular, the set of active indices by
$$\mathcal{V}_j(\bar{x}):=\big\{v_j\in \mathcal{V}_j \bigm| ~G_j(\bar{x})=g_j(\bar{x},v_j)\big\}=
\big\{v_j\in \mathcal{V}_j \bigm| ~g_j(\bar{x},v_j)\geq G_j(\bar{x})\big\}=\mathcal{V}^0_j(\bar{x}).$$

Clearly, $\mathcal{V}_j(\bar{x}) \subseteq \mathcal{V}^\varepsilon_j(\bar{x})$. Furthermore, $\mathcal{V}^\varepsilon_j(\bar{x})\neq \emptyset$, while $\mathcal{V}_j(\bar{x})$ could be empty (Consider $g_j(x,v_j):=v_jx^2$ where $x\in \mathbb{R}$, $v_j\in \mathcal{V}_j:=[0,1)$, and $\bar{x}=1$; See \cite{chu-n}).
In \cite{mor-c}, it has been shown that the perturbed sets of active indices, $\mathcal{V}^\varepsilon_j(\bar{x})$, $j\in J$,
are useful tools in calculating the subdifferentials of a nonconvex supremum function.

As usual in optimization, an appropriate CQ is required to derive KKT optimality conditions.
Here, we present and use a CQ in a nonsmooth setting for  $(P_\mathcal{U})$.

\begin{definition}\label{cq}
We say CQ2 is satisfied at $\bar{x}\in \Omega$ if
$$0 \notin cl^*\, co\, (\bigcup_{j \in J} \big\{\partial_x g_j(\bar{x},v_j),~v_j \in \mathcal{V}_j(\bar{x})\big\}).$$
\end{definition}

In fact, CQ2 is a nonsmooth version of the extended Mangasarian-Fromovitz Constraint Qualification (MFCQ); See \cite[Page 72]{bon}.
Now, we are ready to provide necessary conditions for local highly robust weakly efficient solutions of $(P_\mathcal{U})$.

\begin{theorem}\label{neckkt}
Let $X$ be an Asplund space and $\bar{x}\in\Omega$. Assume that $f$ is continuous at $\bar x$, and for each $j\in J$, the following assumptions hold:
\begin{itemize}
\item[(i)] The set $\mathcal{V}^{\varepsilon_j}_j(\bar{x})$ is compact for some $\varepsilon_j>0$.
\item[(ii)] For each $x$ around $\bar{x}$, the function $g_j(x,\cdot)$ is
upper semicontinuous on $\mathcal{V}^{\varepsilon_j}_j(\bar{x})$ and
the function $g_j(\cdot,v_j)$ is uniformly locally Lipschitz at $\bar{x}$ with modulus $L_j>0$ for each $v_j\in \mathcal{V}_j$.
\item[(iii)] The set-valued mapping $\partial_{x} g_j$ is weak$^*$ closed at $(\bar{x},\bar{v}_j)$ for each
$\bar{v}_j \in \mathcal{V}_j(\bar{x})$.
\item[(iv)] The functions $f_i$ are Sequentially Normally Epi-Compact (SNEC) at $\bar{x}$ (in the sense of \cite[Definition 1.116]{mor-b1}) for all but one $i\in I$ and
$$\bigg[\sum_{i=1}^p x_i^*= 0,~~ x_i^*\in \partial^\infty f_i(\bar{x})\bigg] \Longrightarrow x_i^* = 0,~~\forall i\in I.$$
\end{itemize}
If $\bar{x}$ is a local highly robust weakly efficient solution of $(P_\mathcal{U})$, then
for each $u\in\mathcal{U}$ there exist $\lambda_i \geq 0,~i\in I,$ and $\mu_j \geq 0,~j\in J,$
with $\sum^p_{i=1} \lambda_i + \sum^q_{j=1} \mu_j =1$, such that
\begin{equation}\label{kkt1}
\sum^p_{i=1} \lambda_i u_i \in \sum^p_{i=1} \lambda_i\circ\partial f_i(\bar{x}) +
\sum^q_{j=1} \mu_j cl^* co\, \big\{ \partial_x g_j(\bar{x},v_j) :~ v_j \in \mathcal{V}_j(\bar{x})\big\},
\end{equation}
\begin{equation}\label{kkt2}
\mu_j \sup_{v_j\in \mathcal{V}_j} g_j(\bar{x},v_j) = 0,~~\forall j\in J,
\end{equation}
where $\lambda_i\circ\partial f_i (\bar{x})$ is defined as $\lambda_i\partial f_i (\bar{x})$ when $\lambda_i > 0$ and as
$\partial^\infty f_i (\bar{x})$ when $\lambda_i =0$, for each $i\in I$.
In addition, if CQ2 holds at $\bar{x}$, then $\lambda_i$'s, $i\in I,$ are not all zero.
\end{theorem}
\begin{proof}
The mappings $G_j,~j\in J,$ are well-defined and locally Lipschitz at $\bar{x}$
with modulus $L_j,~j\in J$ (According to assumptions (i) and (ii)).
Without loss of generality, let $\mathcal{B}$ be a neighbourhood of $\bar{x}$ such that
\begin{itemize}
\item[(a)] For each $x\in \mathcal{B}$, the function $g_j(x,\cdot)$ is upper semicontinuous on $\mathcal{V}^{\varepsilon_j}_j(\bar{x})$.
\item[(b)] For each $x,y\in \mathcal{B}$ and for each $v_j\in \mathcal{V}_j$,
\begin{equation}\label{eq1-neckkt}
~~~~\vert g_j(x,v_j)-g_j(y,v_j)\vert \leq L_j\Vert x-y \Vert,~~\textmd{and}~~\vert G_j(x)-G_j(y)\vert \leq L_j\Vert x-y \Vert.
\end{equation}
\item[(c)] The vector $\bar{x}$ is a highly robust weakly efficient solution of $(P_\mathcal{U})$ with robust feasible set $\Omega\cap \mathcal{B}$.
\end{itemize}
Consider $j\in J$.
Since $\mathcal{B}$ is a neighbourhood of $\bar{x}$, for $\alpha>0$ there exists $x\in \mathcal{B}$ such that
$\Vert x-\bar{x}\Vert <\alpha$.
To begin with, we prove for each $x \in \mathcal{B}$ with $\Vert x-\bar{x}\Vert <\frac{\varepsilon_j}{2L_j}$ and for each scalar $t\geq 0$ with $t+2L_j \Vert x -\bar{x}\Vert<\varepsilon_j$, we have
\begin{equation}\label{eq2-neckkt}
\mathcal{V}_j^{t}(x) \subseteq \mathcal{V}_j^{\varepsilon_j}(\bar{x}).
\end{equation}
Let $v_j \in \mathcal{V}_j^{t}(x)$. Then $g_j(x,v_j)\geq G_j(x)-t.$
So, by (\ref{eq1-neckkt}), we get
$$\begin{array}{ll}
G_j(\bar{x})-\varepsilon_j &<G_j(\bar{x})-t -2L_j \Vert x -\bar{x}\Vert\leq G_j(x)+L_j \Vert x -\bar{x}\Vert  -t -2L_j \Vert x -\bar{x}\Vert\vspace*{2mm}\\
&= G_j(x)-t -L_j \Vert x -\bar{x}\Vert\leq g_j(x,v_j) -L_j \Vert x -\bar{x}\Vert\leq g_j(\bar{x},v_j).
\end{array}
$$
So, $v_j \in \mathcal{V}_j^{\varepsilon_j}(\bar{x})$ and thus, (\ref{eq2-neckkt}) is verified.

Now, we justify that
\begin{equation}\label{eq3-neckkt}
\forall d\in X, ~\exists \bar{v}_j \in \mathcal{V}_j(\bar{x}); ~G^\circ_j (\bar{x};d) \leq (g^{\bar{v}_j})^\circ_j(\bar{x};d),
\end{equation}
where $g^{v_j}_j(\cdot) = g_j(\cdot,v_j)$.
Assume that $d\in X\setminus \{0\}$ is arbitrary and constant hereafter.
By Clarke’s generalized directional derivative, there exist sequences $x_n \longrightarrow \bar{x}$ and $t_n \downarrow 0$
such that
$$G^\circ_j (\bar{x};d) = \displaystyle\lim_{n\longrightarrow\infty}\frac{G_j(x_n +t_n d) -G_j(x_n)}{t_n}
= \displaystyle\lim_{n\longrightarrow\infty}\frac{G_j(x_n +t_n d) -t_n^2 -G_j(x_n)}{t_n}.$$
Moreover, for each sufficiently large $n$, there exists $v_{jn} \in \mathcal{V}_j$ such that
\begin{equation}\label{eq31-neckkt}
G_j(x_n +t_n d) -t_n^2 \leq g^{v_{jn}}_j(x_n +t_n d)=g_j(x_n +t_n d,v_{jn}),
\end{equation}
(i.e., $v_{jn} \in \mathcal{V}_j^{t^2_n}(x_n +t_n d)$). Hence,
\begin{equation}\label{eq4-neckkt}
\begin{array}{l}
\hspace{-2mm}G^\circ_j (\bar{x};d) \hspace{-1mm}\leq\hspace{-1mm} \displaystyle\lim_{n\longrightarrow\infty}\frac{g_j(x_n +t_n d,v_{jn})-g_j(x_n,v_{jn})}{t_n}= \displaystyle\lim_{n\longrightarrow\infty}\frac{g^{v_{jn}}_j(x_n +t_n d)-g^{v_{jn}}_j(x_n)}{t_n}.
\end{array}
\end{equation}
Utilizing the mean value theorem in Asplund spaces for locally Lipschitz function $g^{v_{jn}}_j$ \cite[Corollary 3.51]{mor-b1},
there are $y_n\in (x_n +t_n d, x_n)$ and $x^*_n\in \partial g^{v_{jn}}_j(y_n)= \partial_x g_j(y_n,v_{jn})$ such that
\begin{equation}\label{eq5-neckkt}
g^{v_{jn}}_j(x_n +t_n d)-g^{v_{jn}}_j(x_n) \leq\langle x^*_n, t_n d\rangle.
\end{equation}
On account of  (\ref{eqs}) and (\ref{eq1-neckkt}), $\Vert x^*_n\Vert \leq L_j$ for all $n\in \mathbb{N}$ sufficiently large; and
since $X$ is an Asplund space, the weak$^*$ sequential compactness property of bounded sets yields $x^*_n \overset{w^*}{\longrightarrow} \bar{x}^*,$ for some
$\bar{x}^*\in X^*$ (without loss of generality, by taking a subsequence if necessary). Due to (\ref{eq2-neckkt}), $\{v_{jn}\}_n \subseteq \mathcal{V}_j^{\varepsilon_j}(\bar{x})$.
So, $v_{jn}$ has a convergent subnet, because of compactness assumption of $\mathcal{V}_j^{\varepsilon_j}(\bar{x})$; See assumption (i).
Without loss of generality, let us assume $v_{jn}\longrightarrow \bar{v}_j$ for some
$\bar{v}_j \in \mathcal{V}_j^{\varepsilon_j}(\bar{x})$ when $n\longrightarrow \infty$.
Therefore, taking superior limit on the both sides of (\ref{eq31-neckkt}) and using properties (a) and (b), we get
$$G_j(\bar{x}) =\displaystyle\limsup_{n\longrightarrow\infty} \big(G_j(x_n +t_n d) -t_n^2\big) \leq
\displaystyle\limsup_{n\longrightarrow\infty} g_j(x_n +t_n d,v_{jn}) \leq g_j(\bar{x},\bar{v}_j).$$
Thus, $\bar{v}_j \in \mathcal{V}_j (\bar{x})$ and $\mathcal{V}_j (\bar{x})\not = \emptyset$.

Now, given that $(y_n,v_{jn}) \longrightarrow (\bar{x},\bar{v}_j)$, $\bar{v}_j \in \mathcal{V}_j (\bar{x})$, $x^*_n \overset{w^*}{\longrightarrow} \bar{x}^*$, and the set-valued mapping $\partial_{x} g_j$ is weak$^*$ closed at $(\bar{x},\bar{v}_j)$ (assumption (iii)),
$\bar{x}^*\in \partial_{x} g_j(\bar{x},\bar{v}_j)=\partial g^{\bar{v}_j}_j(\bar{x})\subseteq \partial^C g^{\bar{v}_j}_j(\bar{x})$.
Combining (\ref{eq4-neckkt}) and (\ref{eq5-neckkt}) justifies (\ref{eq3-neckkt}) as follows:
$$G_j^\circ(\bar{x};d)\leq \lim_{n\longrightarrow\infty} \langle x^*_n ,d\rangle = \langle \bar{x}^* ,d\rangle \leq
(g^{\bar{v}_j}_j)^\circ (\bar{x};d).$$

In what follows, we show
\begin{equation}\label{eq6-neckkt}
\partial G_j(\bar{x}) \subseteq  cl^*co \bigg(\bigcup_{v_j \in \mathcal{V}_j(\bar{x})} \partial_x g_j(\bar{x},v_j)\bigg).
\end{equation}
On the contrary, suppose that $x^* \in \partial G_j(\bar{x})$ while
$x^*\notin cl^*co \bigg(\bigcup_{v_j \in \mathcal{V}_j(\bar{x})} \partial_x g_j(\bar{x},v_j)\bigg).$
Invoking a strict separation theorem \cite[Theorem 3.9]{con}, there exists $\bar{d}\in X\setminus \{0\}$ such that
for each $y^* \in \partial_x^C g_j(\bar{x},v_j)=cl^* co\, \partial_x g_j(\bar{x},v_j)$ and for each $v_j\in \mathcal{V}_j(\bar{x})$, we have $\langle y^* , \bar{d}\rangle < \langle x^* , \bar{d}\rangle.$
Then, according to \cite[Proposition 2.1.2]{cla-o}, for each $v_j\in \mathcal{V}_j(\bar{x})$,
$$(g^{v_j}_j)^\circ (\bar{x};\bar{d}) = \max \big\{\langle y^* , \bar{d}\rangle : ~y^* \in cl^* co\, \partial_x g_j(\bar{x},v_j)\big\}
< \langle x^* , \bar{d}\rangle \leq G_j^\circ(\bar{x};\bar{d}),$$
which contradicts (\ref{eq3-neckkt}).
Hence, Equation (\ref{eq6-neckkt}) is established.

To continue the proof, we consider the uncertain single-objective optimization problem
$(\bar{P}_\mathcal{U}):=(\bar{P}_u : u\in \mathcal{U})$ as
a family of parameterized problems
$$\begin{array}{l}
(\bar{P}_{u}):~~~\min \mathcal{H}(x,u)~~s.t.~~x\in X,
\end{array}
$$
where $\mathcal{H}:X\times\mathcal{U}\longrightarrow \mathbb{R}$ is a parametric real-valued function
defined by
$$\mathcal{H}(x,u):= \max_{i\in I,~j\in J} \{f_i(x,u)-f_i(\bar{x},u),G_j(x)\},~~\textmd{for each}~~(x,u)\in X\times\mathcal{U}.$$
We claim that $\bar{x}$ is a local highly robust optimal solution of $(\bar{P}_\mathcal{U})$ (with feasible solution set $\mathcal{B}$).
By arguing indirectly, assume that there are $\hat{x}\in\mathcal{B}$ and $\hat{u}\in\mathcal{U}$ such that
$\mathcal{H}(\hat{x},\hat{u})<\mathcal{H}(\bar{x},\hat{u})=0.$
Then,
$$\left\lbrace\begin{array}{ll}
f_i(\hat{x},\hat{u})-f_i(\bar{x},\hat{u})<0,~~&\forall i\in I,\vspace*{1mm}\\
G_j(\hat{x})<0,~~&\forall j \in J.
\end{array}\right.
~~\Longrightarrow ~~\left\lbrace\begin{array}{ll}
f(\hat{x},\hat{u})<f(\bar{x},\hat{u}),\vspace*{1mm}\\
\hat{x}\in \Omega\cap\mathcal{B},
\end{array}\right.
$$
which contradicts (c).
So, $\bar{x}$ is a local highly robust optimal solution of $(\bar{P}_\mathcal{U})$.
Now, by applying a nonsmooth version of Fermat's rule \cite[Proposition 1.114]{mor-b1}, we get $0\in \partial_x \mathcal{H}(\bar{x},u)$ for each
$u\in \mathcal{U}$.
So, for each $u\in \mathcal{U}$, there exist $\lambda_i \geq 0,~i\in I,$ and $\mu_j \geq 0,~j\in J,$
with $\sum^p_{i=1} \lambda_i + \sum^q_{j=1} \mu_j =1$, such that
$$ \left\lbrace\begin{array}{l}
0\in \sum^p_{i=1} \lambda_i\circ\partial_x f_i(\bar{x},u) +\sum^q_{j=1} \mu_j \partial G_j(\bar{x}),\vspace*{1mm}\\
\mu_j G_j(\bar{x}),~~\forall j\in J,
\end{array}\right.
$$
by \cite[Theorem 3.46 (i)]{mor-b1} in Asplund spaces and assumption (iv).
Finally, invoking \cite[Proposition 1.107 (ii) and (iii)]{mor-b1} and (\ref{eq6-neckkt}), we get
$$\begin{array}{c}
\displaystyle\sum^p_{i=1}\lambda_i u_i \in \sum^p_{i=1} \lambda_i\circ\partial f_i(\bar{x}) +\sum^q_{j=1} \mu_j cl^*co \big\{ \partial_x g_j(\bar{x},v_j): ~ v_j \in \mathcal{V}_j(\bar{x})\big\},\vspace*{0mm}\\
\mu_j \displaystyle\sup_{v_j\in \mathcal{V}_j} g_j(\bar{x},v_j) = 0,~~\forall j\in J.
\end{array}$$
The proof of the first part of the theorem is complete. To prove the second part, assume CQ2.
By indirect proof, assume that all $\lambda_i$'s, $i\in I$, in Equation (\ref{kkt1}) are zero.
Then, $\sum^q_{j=1} \mu_j =1$ and
$0 \in \sum^q_{j=1} \mu_j cl^*co \big\{ \partial_x g_j(\bar{x},v_j): ~ v_j \in \mathcal{V}_j(\bar{x})\big\},$
which contradicts CQ2.
The proof is complete.
\end{proof}

\begin{remark}
There are several points about the assumptions (i)-(iv) in Theorem \ref{neckkt}.
These assumptions have broadly been applied to calculate nonsmooth subdifferentials/subgradients of supremum/max mappings in
nonsmooth analysis; See \cite{chu-o,chu-r,chu-n,mor-b1,mor-c,zhe} and the references therein.
In derails, hypothesises (i) and (ii) guarantee definability and locally Lipschitzness of the functions $G_j,~j\in J$.
Hypothesises (iii), 
is an extension of a condition existing in the literature for finite-dimensional spaces.
This property occurs for a large class of functions, for example, for uniformly sequential (lower) regular functions \cite{chu-n}, subsmooth functions \cite{zhe}, and continuously prox-regular functions
\cite{lev} when (i) and (ii) hold.
Moreover, (iv) ensures that the subdifferential of max mapping is a subset of that of the weighted sum mapping.
Under locally Lipschitzness of all but one of functions $f_i,~i\in I$, (iv) automatically holds due to
\cite[Theorem 1.26 and corollary 1.81]{mor-b1} and as a result of \cite[Theorem 1.26]{mor-b1}; See also \cite[Page 121]{mor-b1}.
\end{remark}

The following example illustrates Theorem \ref{neckkt}.


\begin{example}\label{ex-neckkt}
Consider the objective functions $f_i,~i=1,2$, and the uncertainty set $\mathcal{U}$ with $\mathcal{U}_1 = \mathcal{U}_2 := [-1, 0]\times [-1, 0] \subseteq \mathbb{R}^2$ (given in Example \ref{ex1-nec1}) as $f_1(x_1, x_2) := \vert x_1 + 1\vert + x_2$ and $f_2(x_1, x_2) := x_1 + \vert x_2 + 1\vert.$
Also, consider the constraint functions $g_j:\mathbb{R}^2\times\mathcal{V}_j \longrightarrow \mathbb{R},~j=1,2,3,$ as
$$\begin{array}{l}
g_1(x,v_1):=x_1\sin(v_1)+x_2\cos(v_1)-1,\vspace*{0.5mm}\\
g_2(x,v_2):=-x_1-v_2,~~~g_3(x,v_3):=-x_2-1+v_3,
\end{array}
$$
where $v_1\in \mathcal{V}_1:=[-\frac{\pi}{2},\pi]$, $v_2\in \mathcal{V}_2:=[1,2)$, and $v_3\in \mathcal{V}_3:=(-1,0]$.
It is not difficult to see that $\bar{x}:=(-1,-1)\in \Omega\subseteq [-1,1]\times [-1,1]$.
So, due to Example \ref{ex1-nec1}, $(-1,-1)$ is a highly robust efficient solution of $(P_\mathcal{U})$.\\
Obviously, the hypothesises (i)-(iv) in Theorem \ref{neckkt} hold.
Now, straightly calculation gives us that $\mathcal{V}_1(\bar{x})=\{-\frac{\pi}{2},\pi\}$, $\mathcal{V}_2(\bar{x})=\{1\}$,
$\mathcal{V}_3(\bar{x})=\{0\}$, and
$$\begin{array}{c}
\partial f_1(\bar{x}) = [-1, 1]\times \{1\},~~~~ \partial f_2(\bar{x}) = \{1\}\times [-1, 1]\vspace*{1mm}\\
\{ \partial_x g_1(\bar{x},v_1): ~ v_1 \in \mathcal{V}_1(\bar{x})\}=\{(-1,0),(0,-1)\},\vspace*{1mm}\\
\{ \partial_x g_2(\bar{x},v_2): ~ v_1 \in \mathcal{V}_2(\bar{x})\}=\{(-1,0)\},~~\{ \partial_x g_3(\bar{x},v_3): ~ v_1 \in \mathcal{V}_3(\bar{x})\}=\{(0,-1)\}.
\end{array}$$
Therefore, for each $u=(u_1,u_2)\in \mathcal{U}=\mathcal{U}_1\times\mathcal{U}_2$, by setting
$\gamma :=5-(u_{11}+u_{12}+u_{21}+u_{22})$, $\lambda_1=\lambda_2=\mu_1:=\frac{1}{\gamma}$,
$\mu_2:=\frac{1-u_{11}-u_{12}}{\gamma}$, and $\mu_3:=\frac{1-u_{21}-u_{22}}{\gamma}$,
(\ref{kkt1}) and (\ref{kkt2}) are satisfied.
\end{example}

The following corollary is a direct consequence of Theorem \ref{neckkt}.

\begin{corollary}\label{co1-neckkt}
Let $X$ be an Asplund space and $\bar{x}\in\Omega$.
Let assumptions (i)-(iii) hold at $\bar{x}$, for each $j\in J$, and
the functions $f_i$ be locally Lipschitz at $\bar{x}$, for all but one $i\in I$.
If $\bar{x}$ is a local highly robust weakly efficient solution of $(P_\mathcal{U})$, then
for each $u\in\mathcal{U}$ there exist $\lambda_i \geq 0,~i\in I,$ and $\mu_j \geq 0,~j\in J,$
with $\sum^p_{i=1} \lambda_i + \sum^q_{j=1} \mu_j =1$, such that
\begin{equation}\label{kkt3}
\displaystyle\sum^p_{i=1} \lambda_i u_i \in \sum^p_{i=1} \lambda_i\partial f_i(\bar{x}) +
\sum^q_{j=1} \mu_j cl^* co\, \big\{ \partial_x g_j(\bar{x},v_j) :~ v_j \in \mathcal{V}_j(\bar{x})\big\},
\end{equation}
\begin{equation}\label{kkt4}
\displaystyle\mu_j \sup_{v_j\in \mathcal{V}_j} g_j(\bar{x},v_j) = 0,~~\forall j\in J.
\end{equation}
In addition, if CQ2 holds at $\bar{x}$, then $\lambda_i$'s, $i\in I,$ are not all zero.
\end{corollary}

Corollaries \ref{co2-neckkt} and \ref{co3-neckkt} are two versions of Theorem \ref{neckkt} in general real Banach spaces with Clarke's generalized gradients.

\begin{corollary}\label{co2-neckkt}
Let $X$ be a real Banach space and $\bar{x}\in\Omega$.
Let assumptions (i)-(iii) (with Clarke's generalized gradient) hold at $\bar{x}$, for each $j\in J$, and
the functions $f_i$ be locally Lipschitz at $\bar{x}$, for each $i\in I$.
If $\bar{x}$ is a local highly robust weakly efficient solution of $(P_\mathcal{U})$, then
for each $u\in\mathcal{U}$ there exist $\lambda_i \geq 0,~i\in I,$ and $\mu_j \geq 0,~j\in J,$
with $\sum^p_{i=1} \lambda_i + \sum^q_{j=1} \mu_j =1$, such that
$$
\begin{array}{c}
\displaystyle\sum^p_{i=1} \lambda_i u_i \in \sum^p_{i=1} \lambda_i\partial^C f_i(\bar{x}) +
\sum^q_{j=1} \mu_j cl^* co\, \big\{ \partial_x^C g_j(\bar{x},v_j) :~ v_j \in \mathcal{V}_j(\bar{x})\big\},\vspace*{1mm}\\
\displaystyle\mu_j \sup_{v_j\in \mathcal{V}_j} g_j(\bar{x},v_j) = 0,~~\forall j\in J.
\end{array}
$$
In addition, if CQ2 holds at $\bar{x}$, then $\lambda_i$'s, $i\in I,$ are not all zero.
\end{corollary}

\begin{corollary}\label{co3-neckkt}
Let $X$ be a real Banach space and $\bar{x}\in\Omega$.
For each $j\in J$, let assumptions (i)-(iii) (with Clarke's generalized gradient) hold at $\bar{x}$,
$\mathcal{V}_j$ be convex, $g_j(\cdot,v_j)$ be Clarke regular at $\bar{x}$ for each $v_j \in \mathcal{V}_j$,
and $g_j(x,\cdot)$ be concave on $\mathcal{V}_j$ for $x$ around $\bar{x}$.
Furthermore, suppose that the functions $f_i$ are locally Lipschitz at $\bar{x}$ for each $i\in I$.
If $\bar{x}$ is a local highly robust weakly efficient solution of $(P_\mathcal{U})$, then
for each $u\in\mathcal{U}$ there exist $\lambda_i \geq 0,~i\in I,$ and $\mu_j \geq 0,~j\in J,$
with $\sum^p_{i=1} \lambda_i + \sum^q_{j=1} \mu_j =1$ and $v_j\in\mathcal{V}_j(\bar{x}),~j\in J,$ such that
$$
\displaystyle\sum^p_{i=1} \lambda_i u_i \in \sum^p_{i=1} \lambda_i\partial^C f_i(\bar{x}) +
\sum^q_{j=1} \mu_j \partial_x^C g_j(\bar{x},v_j),~~\textmd{and}~~\mu_j g_j(\bar{x},v_j) = 0,~~\forall j\in J.
$$
In addition, if $0\notin \big\{\partial_x^C g_j(\bar{x},v_j): v_j\in\mathcal{V}_j(\bar{x}),~j\in J\big\}$, then $\lambda_i$'s, $i\in I,$ are not all zero.
\end{corollary}
\begin{proof}
It is sufficient to demonstrate that $\digamma_j:=\big\{\partial_x^C g_j(\bar{x},v_j): v_j\in\mathcal{V}_j(\bar{x})\big\}$ is weak$^*$ closed and convex for each $j\in J$.
Then, the desired result is derived by applying Corollary \ref{co2-neckkt}. The set of points at which a concave function attains its supremum on a convex set, is a convex set.
So, $\mathcal{V}_j(\bar{x})$ is a convex set for each $j\in J$.
On the other hand, $\mathcal{V}_j(\bar{x})$ is a closed subset of the compact set $\mathcal{V}_j^{\epsilon_j}(\bar{x})$ (assumption (i)). Therefore, it is also a compact set. First, we justify $\digamma_j$ is a convex set.
According to the properties of Clarke's generalized gradient, $\partial_x^C g_j(\bar{x},v_j)$ is a convex set \cite{cla-f}.
Assume that $\hat{v}_j, \tilde{v}_j\in \mathcal{V}_j(\bar{x})$, $x_1^* \in \partial_x^C g_j(\bar{x},\hat{v}_j)$,
$x_1^* \in \partial_x^C g_j(\bar{x},\tilde{v}_j)$, and $0<\gamma<1$.
We prove that $$\bar{x}^*:=\gamma x_1^* +(1-\gamma)x^*_2 \in \partial_x^C g_j(\bar{x},\bar{v}_j),$$
where $\bar{v}_j :=\gamma\hat{v}_j+(1-\gamma)\tilde{v}_j$.
Using the regularity and concavity assumption and definition of Clarke's generalized gradient, for each $d\in X$, we get
$$\begin{array}{ll}
\langle \bar{x}^* ,d\rangle &= \gamma\langle x_1^*,d\rangle +(1-\gamma)\langle x^*_2, d\rangle
\leq \gamma (g_j^{\hat{v}_j})^\circ(\bar{x};d) + (1-\gamma) (g_j^{\tilde{v}_j})^\circ(\bar{x};d)\vspace*{2mm}\\
&= \gamma \displaystyle\lim_{t\downarrow 0} \frac{g_j(\bar{x}+td,\hat{v}_j)-g_j(\bar{x},\hat{v}_j)}{t} + (1-\gamma) \lim_{t\downarrow 0} \frac{g_j(\bar{x}+td,\tilde{v}_j)-g_j(\bar{x},\tilde{v}_j)}{t}\vspace*{2mm}\\
&=\displaystyle\lim_{t\downarrow 0} \frac{\big(\gamma g_j(\bar{x}+td,\hat{v}_j)+(1-\gamma)g_j(\bar{x}+td,\tilde{v}_j)\big)-
\big(\gamma g_j(\bar{x},\hat{v}_j)+(1-\gamma)g_j(\bar{x},\tilde{v}_j)\big)}{t}\vspace*{2mm}\\
&\leq \displaystyle\lim_{t\downarrow 0} \frac{g_j(\bar{x}+td,\bar{v}_j)-g_j(\bar{x},\bar{v}_j)}{t}= (g_j^{\bar{v}_j})^\circ(\bar{x};d).
\end{array}$$
Thus, $\bar{x}^* \in \partial_x^C g_j(\bar{x},\bar{v}_j)$ and convexity of $\digamma_j$ is justified.

Finally, to establish the weak$^*$ closedness of $\digamma_j$, we use the fact that the image of a weak$^*$ closed set-valued function on a compact set is a weak$^*$ closed set.
Since $\digamma_j=\partial_x^C g_j(\bar{x},\mathcal{V}_j(\bar{x}))$,
$\mathcal{V}_j(\bar{x})$ is compact, and $\partial_x^C g_j$ is weak$^*$ closed at $(\bar{x},\bar{v}_j)$ for each
$\bar{v}_j\in \mathcal{V}_j(\bar{x})$, thus $\digamma_j$ is weak$^*$ closed, due to Lemma \ref{lem}.
This completes the proof.
\end{proof}

Corollary \ref{co4-neckkt} is a finite-dimensional version of Theorem \ref{neckkt}.

\begin{corollary}\label{co4-neckkt}
Let $X$ be a finite-dimensional Banach space, $\bar{x}\in\Omega$, and assumptions (i)-(iii) hold at $\bar{x}$, for each $j\in J$.
If $\bar{x}$ is a local highly robust weakly efficient solution of $(P_\mathcal{U})$, then
for each $u\in\mathcal{U}$ there exist $\lambda_i \geq 0,~i\in I,$ and $\mu_j \geq 0,~j\in J,$
with $\sum^p_{i=1} \lambda_i + \sum^q_{j=1} \mu_j =1$, such that $\displaystyle\mu_j \sup_{v_j\in \mathcal{V}_j} g_j(\bar{x},v_j) = 0$ for each $j\in J$ and
$$
\begin{array}{c}
\sum^p_{i=1} \lambda_i u_i \in \sum^p_{i=1} \lambda_i\circ\partial f_i(\bar{x}) +
\sum^q_{j=1} \mu_j co\, \big\{ \partial_x g_j(\bar{x},v_j) :~ v_j \in \mathcal{V}_j(\bar{x})\big\},
\end{array}
$$
In addition, if $0\notin co\,\big\{\partial_x g_j(\bar{x},v_j): v_j\in\mathcal{V}_j(\bar{x}),~j\in J\big\}$, then $\lambda_i$'s, $i\in I,$ are not all zero.
\end{corollary}
\begin{proof}
Let us fix $j\in J$.
Set $\Upsilon_j := \big\{ \partial_x g_j(\bar{x},v_j) :~ v_j \in \mathcal{V}_j(\bar{x})\big\}.$
Since $\mathcal{V}_j^{\epsilon_j}(\bar{x})$ is compact (assumption (i)) and
$\mathcal{V}_j(\bar{x}) \subseteq \mathcal{V}_j^{\epsilon_j}(\bar{x})$ is a closed set, by upper semicontinuity of the function $g_j(x,\cdot)$ on $\mathcal{V}^{\varepsilon_j}_j(\bar{x})$ for each $x$ around $\bar{x}$ (assumption (ii)), $\mathcal{V}_j(\bar{x})$ is compact.
Thus, according to Lemma \ref{lem},
$\Upsilon_j =\partial_x g_j(\bar{x},\mathcal{V}_j(\bar{x})),$
is closed utilizing the closedness of $\partial_x g_j$ at $(\bar{x},\bar{v}_j)$ for each
$\bar{v}_j\in \mathcal{V}_j(\bar{x})$ (assumption (iii)).
On the other hand, uniformly local Lipschitzness of $g_j(\cdot,v_j)$ at $\bar{x}$ with modulus $L_j>0$ (assumption (ii))
ensures boundedness of $\Upsilon_j$ by (\ref{eqs}).
Hence, $\Upsilon_j$ is a compact set.
It yields $co\,\Upsilon_j$ is compact as well using \cite[Theorem 3.20]{rud}.
Now, the desired result is obtained from Theorem \ref{neckkt}.
\end{proof}

In the last part of this section, we obtain sufficient conditions for local highly robust weakly/strictly efficient solutions of $(P_\mathcal{U})$.
First, we introduce a highly robust KKT condition.

\begin{definition}\label{kkt}
Given $\bar{x}\in \Omega$, we say that highly robust KKT condition for $(P_\mathcal{U})$ at $\bar{x}$ is satisfied
if for each $u\in\mathcal{U}$ there exist $\lambda_i \geq 0,~i\in I,$ with $\sum^p_{i=1} \lambda_i =1$ and $\mu_j \geq 0,~j\in J,$
such that (\ref{kkt3}) and (\ref{kkt4}) hold.
\end{definition}

It can be seen that highly robust KKT conditions may not yield highly robustness even when the functions
are smooth and the objective functions are not perturbed, $\mathcal{U}=\{0\}$; see \cite[Example 3.8]{chu-o}.
So, to provide a sufficient condition for (local) highly robust weak/strict efficiency,
we introduce a generalized convexity notion at a point on a set for a family of vector- and real-valued functions under uncertainty.

\begin{definition}
Let $\bar{x}\in \Omega$. We say that $(f,g)$ of $(P_\mathcal{U})$ is \\
(i) generalized convex at $\bar{x}$ if for each $x \in \Omega$
and each $u\in \mathcal{U}$ there exists $d\in X$ such that
\begin{equation*}\label{gc1}\begin{array}{l}
\langle x^*_i, d\rangle \leq f_i(x,u_i) - f_i(\bar{x},u_i),~~~\forall x^*_i\in \partial_x f_i(\bar{x},u_i),~\forall i\in I,\vspace*{1mm}\\
\langle x^*_j, d\rangle \leq g_j(x,v_j) - g_j(\bar{x},v_j),~~~\forall x^*_j\in \partial_x g_j(\bar{x},v_j),~\forall v_j\in \mathcal{V}_j(\bar{x}),~\forall j\in J.
\end{array}
\end{equation*}
(ii) strictly generalized convex at $\bar{x}$ if for each $x \in \Omega$
and each $u\in \mathcal{U}$ there exists $d\in X$ such that
\begin{equation*}\label{gc2}\begin{array}{l}
\langle x^*_i, d\rangle < f_i(x,u_i) - f_i(\bar{x},u_i),~~~\forall x^*_i\in \partial_x f_i(\bar{x},u_i),~\forall i\in I,\vspace*{1mm}\\
\langle x^*_j, d\rangle \leq g_j(x,v_j) - g_j(\bar{x},v_j),~~~\forall x^*_j\in \partial_x g_j(\bar{x},v_j),~\forall v_j\in \mathcal{V}_j(\bar{x}),~\forall j\in J.
\end{array}
\end{equation*}
If the conditional parts of the above four inequalities occur for each $x\in \Omega\cap\mathcal{B}$ for some neighbourhood $\mathcal{B}$
of $\bar{x}$ instead of $\Omega$, then this concepts are defined locally.
\end{definition}

The aforementioned generalized convexity notion is a type of convexity defined in \cite{chu-o}. 
It is clear that (local) convexity implies generalized (local) convexity.
Generalized (locally) convex functions at a given point on a
set form a large collection rather than (locally) convex functions; see \cite[Example 3.10]{chu-o}.

\begin{theorem}\label{sufkkt}
Let the highly robust KKT conditions be satisfied at $\bar{x}\in \Omega$. If $(f,g)$ is generalized locally convex at $\bar{x}$, then $\bar{x}$ is a local highly robust weakly efficient solution of $(P_\mathcal{U})$.
\end{theorem}
\begin{proof}
Set
$\Upsilon_j :=\big\{ \partial_x g_j(\bar{x},v_j) :~ v_j \in \mathcal{V}_j(\bar{x})\big\}$ for $j\in J.$
First, we show that if
$$\langle y^*_j, d\rangle \leq g_j(x,v_j) - g_j(\bar{x},v_j),~~~~\forall y^*_j\in \partial_x g_j(\bar{x},v_j),~\forall v_j\in \mathcal{V}_j(\bar{x}),~\forall j\in J,$$
for some $x\in \Omega$ and $d\in X$, then
\begin{equation}\label{eq0-sufkkt}
\langle y^*_j, d\rangle \leq G_j(x) - G_j(\bar{x}),~~~~\forall y^*_j\in cl^*co\,\Upsilon_j,~\forall j\in J.
\end{equation}
Let $y^*_j\in co\,\Upsilon_j$.
Then, there exist $n\in \mathbb{N}$, $\lambda_{jk} \geq 0,~k=1,2,\ldots, n,$
with $\sum_{k=1}^n \lambda_{jk}=1$ and $y^*_{jk}\in \partial_x g_j(\bar{x},v_{jk}),~k=1,2,\ldots, n,$ with
$v_{jk}\in \mathcal{V}_j(\bar{x})$ such that
$y^*_j = \sum_{k=1}^n \lambda_{jk} y^*_{jk}.$
We get
$$
\begin{array}{l}
\langle y^*_j, d\rangle = \langle \displaystyle\sum_{k=1}^n \lambda_{jk} y^*_{jk}, d\rangle \leq \displaystyle\sum_{k=1}^n \lambda_{jk} \big(g_j(x,v_{jk}) - g_j(\bar{x},v_{jk})\big)\leq G_j(x) - G_j(\bar{x}),
\end{array}
$$
noting that $g_j(x,v_{jk}) \leq G_j(x)$ and $g_j(\bar{x},v_{jk})=G_j(\bar{x})$ for each $k$.
Now, let $y^*_j\in cl^* co\,\Upsilon_j$.
Then, there exists a net $\{y^*_{j\alpha}\}_{\alpha} \subseteq co\,\Upsilon_j$ with
$y^*_{j\alpha} \overset{w^*}{\longrightarrow} y^*_j$ as $\alpha \longrightarrow\infty$.
In this situation,
$\langle y^*_{j\alpha}, d\rangle \leq G_j(x) - G_j(\bar{x})$ for each $\alpha$. Hence, by $\alpha\rightarrow \infty$, we have $\langle y^*_j, d\rangle \leq G_j(x) - G_j(\bar{x})$, and (\ref{eq0-sufkkt}) is valid.

Let $(f,g)$ be generalized convex at $\bar{x}$ on $\Omega \cap \mathcal{B}$
for some neighbourhood $\mathcal{B}$ of $\bar{x}$, $\hat{x} \in\Omega\cap\mathcal{B},$ and $\hat{u}\in\mathcal{U}$.
According to the definition of the highly robust KKT conditions, there exist $\lambda_i \geq 0,~i\in I,$
with $\sum^p_{i=1} \lambda_i =1$, $\mu_j \geq 0,~j\in J,$ $\bar{x}^*_i \in \partial f_i(\bar{x}),~i\in I,$ and
$\bar{y}^*_j\in cl^* co\, \Upsilon_j,~j\in J,$ such that
\begin{equation}\label{eq1-sufkkt}
\sum^p_{i=1} \lambda_i \hat{u}_i=\sum^p_{i=1} \lambda_i \bar{x}^*_i + \sum^q_{j=1} \mu_j \bar{y}^*_j,
~~\textmd{and}~~
\displaystyle\mu_j \sup_{v_j\in \mathcal{V}_j} g_j(\bar{x},v_j) = 0,~~\forall j\in J.
\end{equation}
On account of generalized convexity of $(f,g)$ at $\bar{x}$ on $\Omega \cap \mathcal{B}$, \cite[Proposition 1.107]{mor-b1},
and (\ref{eq0-sufkkt}), there is some $d\in X$ that
\begin{equation}\label{eq2-sufkkt}
\langle \bar{x}^*_i - \hat{u}_i ,d\rangle \leq f_i(\hat{x},\hat{u}_i)-f_i(\bar{x},\hat{u}_i),~~~\forall i\in I,
\end{equation}
\begin{equation}\label{eq3-sufkkt}
\langle \bar{y}^*_j ,d\rangle \leq G_j(\hat{x})-G_j(\bar{x}),~~~\forall j\in J.
\end{equation}
Combining (\ref{eq1-sufkkt}), (\ref{eq2-sufkkt}), and (\ref{eq3-sufkkt}), we get
$$\begin{array}{l}
\displaystyle\sum^p_{i=1} \lambda_i f_i(\bar{x},\hat{u}_i) \leq \displaystyle\sum^p_{i=1} \lambda_i f_i(\hat{x},\hat{u}_i) -
\sum^p_{i=1} \lambda_i \langle \bar{x}^*_i - \hat{u}_i ,d\rangle\\
~~~~~~~~= \displaystyle\sum^p_{i=1} \lambda_i f_i(\hat{x},\hat{u}_i) + \sum^q_{j=1} \mu_j \langle \bar{y}^*_j ,d\rangle\leq \displaystyle\sum^p_{i=1} \lambda_i f_i(\hat{x},\hat{u}_i) + \sum^q_{j=1} \mu_j \big(G_j(\hat{x})-G_j(\bar{x}) \big).
\end{array}
$$
By KKT condition (\ref{eq1-sufkkt}), $\mu_j G_j(\bar{x})=0,~j\in J,$ and as $\hat{x}\in \Omega$, $G_j(\hat{x})\leq 0,~j\in J$.
So, we obtain
$\sum^p_{i=1} \lambda_i f_i(\bar{x},\hat{u}_i) \leq \sum^p_{i=1} \lambda_i f_i(\hat{x},\hat{u}_i).$
Hence, for each $u\in \mathcal{U}$, the vector $\bar{x}$ is a locally optimal solution of
$\begin{array}{l}
\min \sum^p_{i=1} \lambda_i f_i(x,u_i)~~s.t.~~x\in \Omega.
\end{array}$
This implies that $\bar{x}$ is a local weakly efficient solution of $(P_u)$ for each $u\in \mathcal{U}$ by \cite[Proposition 3.9]{ehr-b}.
Thus, $\bar{x}$ is a local highly robust weakly efficient solution of $(P_\mathcal{U})$ and the proof is complete.
\end{proof}

A result similar to Theorem \ref{sufkkt} can be obtained for local highly robust strictly efficient solutions under strict generalized local convexity. Furthermore, analogous to Theorem \ref{sufkkt}, a sufficient condition for \textit{global} highly robustness can be derived.

\section{Highly Robustness With Special Uncertain Sets}\label{section5}

In this section, we study highly robustness for an uncertain multi-objective problem $(P_\mathcal{U})$ with a special uncertain set, including ellipsoidal, ball, or a polyhedral set.


\subsection{Ellipsoidal And Ball Uncertain Sets}\label{sub1-section4}

Let the uncertain set $\mathcal{U}_i,~i\in I,$ be an ellipsoidal or ball with  $0\in int\, \mathcal{U}_i,~i\in I$.
The results of this subsection are valid for any uncertain set with the origin as an interior point.

Let $d\in X\setminus \{0\}$ be constant.
Define subspace $Y:=\{\alpha d : \alpha \in \Bbb R\}$ of $X$ and the continuous linear functional
$h(\alpha d):=\alpha\Vert d\Vert$ on space $Y$.
Invoking a consequence of Hahn-Banach theorem \cite[Corollary 6.5]{con}, there exists a continuous linear functional $x^*\in X^*$ where
$\langle x^*,d\rangle=\Vert d\Vert > 0$. So, due to $0\in int\, \mathcal{U}_i,~i\in I,$ without loss of generality,
one can get $x^* \in\mathcal{U}_i,~i\in I$ with $\langle x^*,d\rangle>0$.
Thus,
\begin{equation}\label{eq-sub1}
\forall d\in X\setminus \{0\}, ~\exists u \in \mathcal{U};~ \langle u_i,d\rangle > 0,~\forall i\in I.
\end{equation}
In this situation, due to Proposition \ref{hrwe-se}, each highly robust weakly efficient solution of $(P_\mathcal{U})$
is a strictly efficient solution of $(\ref{prb})$.
Furthermore, the conditions (\ref{eq1-h-pf}) and (\ref{eq-h-pi}) in Theorems \ref{h-pf} and \ref{h-pi}, respectively, can be omitted.

The following theorem gives us relationships between highly robust efficient solutions and isolated efficient solutions.

\begin{theorem}
Let $\bar{x}\in\Omega$ and $\mathbb{B}(0;L)\subseteq \mathcal{U}_i,~i\in I$ for some $L>0$.
\begin{itemize}
\item[i)] If $\bar{x}$ is an isolated efficient solution of $(\ref{prb})$ with constant $L$,
then $\bar{x}$ is a highly robust strictly efficient solution of $(P_{\bar{\mathcal{U}}})$ with
$\bar{\mathcal{U}}_i:= \mathbb{B}(0;L),~i\in I$.
\item[ii)] If $X$ is a reflexive Hilbert space and $\bar{x}$ is a highly robust weakly efficient solution of $(P_\mathcal{U})$,
then $\bar{x}$ is an isolated efficient solution of $(\ref{prb})$ with any constant less than $L$.
\end{itemize}
\end{theorem}
\begin{proof}
The first part is derived from Theorem \ref{i-h}.
To prove part (ii), let $\bar{L}\in (0,L)$ be arbitrary and constant.
For each $x\in \Omega \setminus \{\bar{x}\}$, set $u_i:=\frac{\bar{L}}{\Vert x-\bar{x}\Vert}(x-\bar{x}),~i\in I$.
As $X$ is a reflexive Hilbert space, $u=(u_1,u_2,\ldots,u_p)\in \bar{\mathcal{U}}\subseteq \mathcal{U}$
and $\langle u_i, x-\bar{x}\rangle =\bar{L} \Vert x-\bar{x}\Vert$.
So, $\bar{x}$ is an isolated efficient solution of $(\ref{prb})$ with constant $\bar{L}$, due to Theorem \ref{h-i}.
\end{proof}


\begin{theorem}\label{nec-sub1}
Let $0\in int\, \mathcal{U}_i,~i\in I$ and $X$ be a WCG Asplund space.
Let $\bar{x}\in\Omega$ be a local highly robust weakly efficient solution of $(P_\mathcal{U})$
and the functions $f_i,~i\in I$, be locally Lipschitz at $\bar{x}$.
Then,
$$T_w(\bar{x};\Omega) \bigcap \big\{d\in X : \langle x^*, d\rangle\leq 0,~\forall x^*\in \partial f_i(\bar{x}),
~\forall i\in I\big\}=\{0\}.$$
\end{theorem}
\begin{proof}
Apply Theorem \ref{nec1} accompanying (\ref{eq-sub1}).
\end{proof}

\begin{theorem}\label{suf-sub1}
Let $0\in int\, \mathcal{U}_i,~i\in I$.
Let $X$ be a reflexive space, $\bar{x}\in\Omega$, and
$0\not \in \{ d\in X : \exists x\in X\setminus \{0\}; ~ \frac{x}{\Vert x\Vert} \overset{w}{\longrightarrow} d\}.$
Let there exist no $d\in X\setminus \{0\}$ such that
$$\langle x^*,d\rangle\leq 0,~~~\forall x^*\in \partial f_i(\bar{x}),~~\forall i\in I.$$
\begin{itemize}
\item[i)] If $f$ is generalized locally convex at $\bar{x}$ on $\Omega$, then
$\bar{x}$ is a local highly robust weakly efficient solution of $(P_{\bar{\mathcal{U}}})$
for some ${\bar{\mathcal{U}}} \subseteq\mathcal{U}$.
\item[ii)] If $f$ is strictly generalized locally convex at $\bar{x}$ on $\Omega$, then
$\bar{x}$ is a local highly robust strictly efficient solution of $(P_{\bar{\mathcal{U}}})$
for some ${\bar{\mathcal{U}}} \subseteq\mathcal{U}$.
\end{itemize}
\end{theorem}
\begin{proof}
We claim that for some ${\bar{\mathcal{U}}} \subseteq\mathcal{U}$ there exist no $d\in X$ and $u\in \bar{\mathcal{U}}$ such that
$$\langle x^*-u_i,d\rangle< 0,~~~\forall x^*\in \partial f_i(\bar{x}),~~\forall i\in I.$$
Then, the desired result is derived by Theorem \ref{suf1}.
Assume that our claim is not valid.
Since $0\in int\, \mathcal{U}_i,~i\in I,$ without loss of generality, there exist sequences
$\{d_n\}\subseteq X$ and $\{u_n\} \subseteq \bigcap_{i=1}^p\mathcal{U}_i$
with $\Vert d_n\Vert =1$ and $u_n\longrightarrow 0$ such that
$\langle x^*-u_n,d_n\rangle< 0,$ for each $x^*\in \bigcup_{i=1}^p\partial f_i(\bar{x}).$
As each bounded sequence in a reflexive space has a weakly convergent subsequence,
by passing to a subsequence if necessary, $d_n\overset{w}{\longrightarrow} d$ for some $d\in X\setminus \{0\}$.
Thus,
$\langle x^*,d\rangle\leq 0,$ for each $x^*\in \bigcup_{i=1}^p\partial f_i(\bar{x}),$
which makes a contradiction.
\end{proof}


\subsection{Polyhedral Uncertain Sets}\label{sub1-section4}

In the current subsection, we assume, for each $i\in I$, the uncertain set $\mathcal{U}_i$ is represented as
the convex hull of finitely many vertex points and finitely many extreme directions in $X^*$, i.e., there exist a vertex set
$\mathcal{U}_i^{ep} :=\{w_1^i, w_2^i, \ldots , w_{s_i}^i\}\subseteq X^*$ and
a direction set $\mathcal{U}_i^{ed} := \{d^i_1, d^i_2, \ldots , d^i_{r_i}\}\subseteq X^*$ such that
\begin{equation}\label{eq0-sub2}
\mathcal{U}_i = co\,\mathcal{U}_i^{ep} + co\,cone\,\mathcal{U}_i^{ed}.
\end{equation}
For each $u_i\in \mathcal{U}_i$,
there exist scalars $\alpha_1^i, \alpha_2^i,\ldots,\alpha_{s_i}^i\geq 0$ with $\sum_{k=1}^{s_i} \alpha_k^i =1$ and
$\beta_1^i, \beta_2^i,\ldots,\beta_{r_i}^i\geq 0$ such that
$u_i = \sum_{k=1}^{s_i} \alpha_k^i w_k^i + \sum_{t=1}^{r_i} \beta_t^i d_t^i.$
When $\mathcal{U}_i$ is bounded, we assume $\mathcal{U}_i^{ed}=\emptyset$.

For each $i\in I$, we define $\mathcal{U}_i^{epd}$ as follows. In fact, $\mathcal{U}_i^{epd}$ is the union of the rays constructed by vertex points and directions of $\mathcal{U}_i^{ep}$.
$$\mathcal{U}^{epd}_i := \big\{w_k^i+\gamma d_t^i : w_k^i\in\mathcal{U}_i^{ep},~ d_t^i \in \mathcal{U}_i^{ed},~
\gamma\geq 0,~ k=1,2,\ldots, s_i,~ t=1,2,\ldots, r_i \big\}.$$
Now, we define
$\mathcal{U}^{ep} := \mathcal{U}^{ep}_1\times \mathcal{U}^{ep}_2\times \ldots \times \mathcal{U}^{ep}_p,$ and $\mathcal{U}^{epd} := \mathcal{U}^{epd}_1\times \mathcal{U}^{epd}_2\times \ldots \times \mathcal{U}^{epd}_p.$




In Theorem \ref{t1-sub1}, we show that highly robustness with uncertain sets $\mathcal{U}^{epd}_i,~i\in I,$ 
is sufficient for highly robustness with uncertain set $\mathcal{U}$.

\begin{theorem}\label{t1-sub1}
$\bar{x}\in \Omega$ is a highly robust efficient solution of $(P_\mathcal{U})$ if and only if
it is a highly robust efficient solution of $(P_{\mathcal{U}^{epd}})$.
\end{theorem}
\begin{proof}
The ``only if" part is trivial as $\mathcal{U}^{epd}\subseteq \mathcal{U}$.
To prove the ``if" part, suppose $\bar{x}$ is a highly robust efficient solution of $(P_{\mathcal{U}^{epd}})$ while it is not a highly robust efficient solution of $(P_\mathcal{U})$.
Then, there exist $\hat{x}\in \Omega\setminus\{\bar{x}\}$ and $\hat{u}\in \mathcal{U}$ so that
\begin{equation}\label{eq1-sub2}
f_i(\hat{x})-f_i(\bar{x}) - \langle\hat{u}_i,\hat{x}-\bar{x}\rangle  \leq 0,~~~\forall i\in I,
\end{equation}
$$f_j(\hat{x})-f_j(\bar{x}) - \langle\hat{u}_j,\hat{x}-\bar{x}\rangle < 0, \textmd{ for some } j \in I.$$
As $\hat{u}\in \mathcal{U}$, for each $i\in I$,
there exist scalars $\alpha_1^i, \alpha_2^i,\ldots,\alpha_{s_i}^i\geq 0$ with $\sum_{k=1}^{s_i} \alpha_k^i =1$ and
$\beta_1^i, \beta_2^i,\ldots,\beta_{r_i}^i\geq 0$ such that
$\hat{u}_i = \sum_{k=1}^{s_i} \alpha_k^i w_k^i + \sum_{t=1}^{r_i} \beta_t^i d_t^i.$
Let $i\in I$ be constant. By combining (\ref{eq1-sub2}) with the last equality, we get
$$
\begin{array}{ll}
&\displaystyle \sum_{k=1}^{s_i} \alpha_k^i \big(f_i(\hat{x})-f_i(\bar{x}) -  \langle w_k^i,\hat{x}-\bar{x}\rangle\big) -
 \sum_{t=1}^{r_i} \beta_t^i \langle d_t^i,\hat{x}-\bar{x}\rangle \leq 0.
\end{array}
$$
This implies there exist $k_i \in \{1,2,\ldots, s_i\}$, $t_i \in \{1,2,\ldots, s_i\}$, and $\gamma_i\geq 0$ such that
$$
\begin{array}{ll}
&f_i(\hat{x})-\langle w_{k_i}^i+\gamma_i d_{t_i}^i,\hat{x}\rangle \leq  f_i(\bar{x}) - \langle w_{k_i}^i+ \gamma_i d_{t_i}^i,\bar{x}\rangle.
\end{array}
$$
Analogously, there exist $k_j \in \{1,2,\ldots, s_j\}$, $t_j \in \{1,2,\ldots, s_j\}$, and $\gamma_j\geq 0$ such that
$$
f_j(\hat{x})-\langle w_{k_j}^j+\gamma_j d_{t_j}^j,\hat{x}\rangle <  f_j(\bar{x}) - \langle w_{k_j}^j+ \gamma_j d_{t_j}^j,\bar{x}\rangle.
$$
Set $\bar{u}_i:=w_{k_i}^i+\gamma_i d_{t_i}^i,~i\in I,$ and $\bar{u}:=(\bar{u}_1,\bar{u}_2,\ldots,\bar{u}_p)$.
Then, $\bar{u}\in \mathcal{U}^{epd}$ and $f(\hat{x},\bar{u})\leq f(\bar{x},\bar{u})$
which contradicts highly robustness of $\bar{x}$ for $(P_{\mathcal{U}^{epd}})$
and the proof is complete.
\end{proof}



\end{document}